\documentclass[
leqno,          
12pt,           
a4paper         
]{amsart}

\usepackage[colorlinks,citecolor=blue,urlcolor=blue]{hyperref}          
\usepackage{courier}                                                    
\usepackage{amssymb, amsmath, amsfonts, amsthm, esint, stackrel}                         
\usepackage{mathtools, cite, enumerate, rotating, framed, lipsum, enumitem}       
\usepackage{fancybox, bbm}                                              
\usepackage[dvipsnames]{xcolor}                                         
\usepackage[polish, english]{babel}                                     
\usepackage[T1]{fontenc}                                                
\usepackage[cp1250]{inputenc}                                           
\usepackage[normalem]{ulem}
\usepackage{yfonts}
\usepackage{mathrsfs}

\DeclareMathAlphabet{\mathpzc}{OT1}{pzc}{m}{it}
 \numberwithin{equation}{section}                        

\newcounter{thmccc}[section]

\newcommand{\thmcount}{thmccc}                 
\newcounter{specialcounter}

\newtheorem{Thm}[\thmcount]{Theorem}
\newtheorem{Sthm}[specialcounter]{Theorem}
\newtheorem{Cor}[\thmcount]{Corollary}
\newtheorem{Lem}[\thmcount]{Lemma}
\newtheorem{Prop}[\thmcount]{Proposition}
\newtheorem{Rem}[\thmcount]{Remark}
\newtheorem{Defn}[\thmcount]{Definition}
\newtheorem{Ex}[\thmcount]{Example}
\newtheorem{Asu}[\thmcount]{Assumption}
\newtheorem{Sol}[\thmcount]{Solution}
\newtheorem*{Thmx}{Theorem}
\newtheorem*{Corx}{Corollary}
\newtheorem*{Lemx}{Lemma}
\newtheorem*{Propx}{Proposition}
\newtheorem*{Remx}{Remark}
\newtheorem*{Defnx}{Definition}
\newtheorem*{Exx}{Example}
\newtheorem*{Asux}{Assumption}
\newtheorem*{Solx}{Solution}

\newtheorem{theorem}{Theorem}
\makeatletter
\newcommand{\settheoremtag}[1]{
  \let\oldthetheorem\thetheorem
  \renewcommand{\thetheorem}{#1}
  \g@addto@macro\endtheorem{
    \addtocounter{theorem}{-1}
    \global\let\thetheorem\oldthetheorem}
  }
\makeatother

\newcommand \eq[1]{\begin{equation} #1 \end{equation}}
\newcommand \eqx[1]{\begin{equation*}  #1 \end{equation*}}
\newcommand \al[1]{\begin{align} #1 \end{align}}
\newcommand \alx[1]{\begin{align*}  #1 \end{align*}}
\renewcommand \sp[1]{\begin{equation} \begin{split} #1 \end{split} \end{equation}}
\newcommand \spx[1]{\begin{equation*} \begin{split} #1 \end{split} \end{equation*}}

\newcommand \ite[1]{\begin{itemize}  #1 \end{itemize}}
\newcommand{\thm}[2]{\begin{Thm} \label{#1} #2 \end{Thm}}

\newcommand{\lem}[2]{\begin{Lem} \label{#1} #2 \end{Lem}}

\newcommand{\cor}[2]{\begin{Cor} \label{#1} #2 \end{Cor}}

\newcommand{\prop}[2]{\begin{Prop} \label{#1} #2 \end{Prop}}

\newcommand{\rem}[2]{\begin{Rem} \label{#1} #2 \end{Rem}}

\newcommand{\defn}[2]{\begin{Defn} \label{#1} #2 \end{Defn}}

\newcommand{\pr}[1]{\begin{proof} #1 \end{proof}}

\newcounter{comcount}

\newcommand{\green}{\textcolor{green}}

\renewcommand{\hline}{\vbox{\hrule width\textwidth height 1pt}\smallskip}

\renewcommand{\a}{\alpha}       \newcommand{\be}{\beta}         \newcommand{\e}{\varepsilon}
\newcommand{\w}{\varphi}                 \newcommand{\de}{\delta}
        \newcommand{\la}{\lambda}

\newcommand{\CC}{\mathbb{C}}

\newcommand{\HH}{\mathbb{H}}  \newcommand{\hh}{\mathcal{H}}

  \newcommand{\kk}{\mathcal{K}}
  \renewcommand{\ll}{\mathcal{L}}
  \newcommand{\mm}{\mathcal{M}}
\newcommand{\NN}{\mathbb{N}}

\newcommand{\RR}{\mathbb{R}}  
  
  \renewcommand{\tt}{\mathcal{T}}

\newcommand{\ZZ}{\mathbb{Z}}  

\newcommand{\supp}{\mathrm{supp}}
\newcommand{\8}{\infty}

\newcommand{\Rn}{{\RR^n}}
\newcommand{\Rd}{{\RR^n}}

\renewcommand{\rm}[1]{\mathrm{#1}}
\newcommand{\wt}[1]{\widetilde{#1}}

\newcommand{\abs}[1]{\left| #1 \right|}
\newcommand{\set}[1]{\left\{ #1 \right\}}
\newcommand{\norm}[1]{\left\| #1 \right\|}

\newcommand{\eee}[1]{\left( #1 \right)}

\newcommand{\BMO}{BMO[\mu,h]}

\newcommand{\LLLL}{L^2_{h^{-1}}(\mu)}

\renewcommand{\hh}{H}

\begin{document}

\title{Hardy spaces meet harmonic weights}

 \author[M. Preisner,    A. Sikora,   and   L.  Yan ]{ Marcin Preisner,   \ Adam Sikora, \ and \ Lixin Yan}
\address{Marcin Preisner, Instytut Matematyczny, Uniwersytet Wroc\l awski, \ pl. Grunwaldzki 2/4, 50-384 Wroc\l aw, Poland }
\email{marcin.preisner@uwr.edu.pl}
\address{
Adam Sikora, Department of Mathematics, Macquarie University, NSW 2109, Australia}
\email{adam.sikora@mq.edu.au}
\address{
Lixin Yan, Department of Mathematics, Sun Yat-sen (Zhongshan) University, Guangzhou, 510275, P.R. China}
\email{mcsylx@mail.sysu.edu.cn
}

\date{\today}

\newcommand{\st}{\sqrt{t}}

\begin{abstract}
   	We investigate the Hardy space $H^1_L$ associated with a self-adjoint operator~$L$ defined in a general setting
	in \cite{Hofmann_Memoirs}.  We	assume that there exists an $L$-harmonic
	non-negative function $h$ such that the semigroup $\exp(-tL)$, after applying the Doob
	transform related to $h$, satisfies the upper and lower Gaussian estimates. Under this assumption we describe an
	illuminating characterisation of  the Hardy space $H^1_L$ in terms of a simple  atomic decomposition associated
	with the $L$-harmonic function $h$. Our approach also yields a natural characterisation of the $BMO$-type space
	corresponding to the operator $L$ and dual to $H^1_L$ in the same circumstances.
	
	The applications include surprisingly wide range of operators, such as:  Laplace operators with Dirichlet  boundary
	conditions on some domains in $\Rd$, Schr\"odinger operators with certain potentials,  and Bessel operators.
\end{abstract}


\maketitle

\section{Introduction and statement of main results}
\setcounter{equation}{0}

\subsection{Background.}\label{ssec11}

The classical notion of Hardy spaces is a mainstream  masterpiece
in the core of harmonic analysis, see for example \cite{Fefferman_Stein, Stein, Stein_Weiss}.
There are several equivalent definitions of the real variable Hardy  space $H^1({\mathbb{R}}^n)$.
For example,  $H^1({\mathbb{R}}^n)$ can be defined   in terms of the maximal function associated with  the heat
 semigroup generated by the Laplace operator $\Delta$ on   $\mathbb R^n$.
Recall that a  locally integrable function $f$ on ${\mathbb R^n}$ is said to be in  $H^1({\mathbb{R}}^n)$  if
\begin{eqnarray}\label{e1.1a}
M_{\Delta} f(x)=\sup_{t>0}\abs{e^{{t}\Delta}f(x)}
\end{eqnarray}
belongs to $  L^1({\mathbb{R}}^n)$. If this is the case, then we set
$$
\|f\|_{H^1({\mathbb{R}}^n)}=\|M_{\Delta} f\|_{L^1({\mathbb{R}}^n)}.
$$

The definition above suggests defining Hardy spaces corresponding to a  general self-adjoint operator $L$ by simply
replacing the standard heat propagator by the semigroup $\exp(-tL)$ in \eqref{e1.1a}. Alternatively one can define $H^1_L$
using the square function approach. The theory of Hardy spaces associated with operators  has attracted a lot of attention in
 last   decades  and  has been a  very active research topic  in harmonic analysis, see for example \cite{Auscher_unpublished,
 Auscher_McIntosh_Russ, Auscher_Russ, Chang_Krantz_Stein,  Duong_Yan, DZ_JFAA, Hofmann_Memoirs, Hofmann_Mayboroda,
  Hofmann_Mayboroda_McIntosh, Song_Yan_2018, Yang_Yang} and the references therein. Very systematic and general theory
  of such Hardy spaces was described in \cite{Hofmann_Memoirs}.  In a more specific situation, such as some classes of
  Schr\"odinger operators,  the Hardy spaces $H^1_L$ were studied also by Dziuba\'nski and Zienkiewicz, see for example
   \cite{DZ_Colloquium, DZ_Revista2, DZ_JFAA}.

In our study we investigate $H^1_L$ in the case, when  there exists an $L$-harmonic
non-negative function $h$ such that  the semigroup $\exp(-tL)$, after applying the Doob transform related to $h$,
 satisfies the upper and lower Gaussian estimates. In this situation we are able to obtain a natural characterisation
 of $H^1_L$ in terms of atomic decompositions in which atoms satisfy
 the cancellation associated {with} the harmonic function $h$.

Recall that one of the most fundamental aspect of the theory of Hardy spaces is the atomic decomposition {theorem}
obtained by Coifman and Latter, see \cite{Coifman_Studia} {for} $n=1$ and
\cite{Latter_Studia} {for} $n\geq 2$.  It is known  that $f\in H^1({\mathbb{R}}^n)$ if and only if
$$f = \sum_{k=1}^\8 \la_k a_k,$$
where $\sum_k|\la_k| <\8$ and $a_k$ are classical atoms, i.e. there exist balls $B_k$ such that
\begin{eqnarray}\label{e1.2a}
\supp \, a_k \subseteq B_k, \qquad \norm{a_k}_\8 \leq |B_k|^{-1}, \qquad  \int_{B_k} a_k(x) \, dx=0.
\end{eqnarray}
Moreover, we can choose $\la_k$'s such that
$$C^{-1} \norm{f}_{H^1({\mathbb{R}}^n)} \leq \sum_{k=1}^\8 |\la_k| \leq C \norm{f}_{H^1({\mathbb{R}}^n)}. $$
The atomic description of Hardy spaces is particularly useful and it is the primary point of interest of this paper. Our main  observation in this study
states that under our assumption involving the Doob transform such characterisation
remains valid with the cancellation part of condition \eqref{e1.2a}  replaced by the relation
\begin{equation}\label{can}
\int_{B_k} a_k(x)h(x) \, dx=0.
\end{equation}

Another fundamental aspect of classical theory of Hardy spaces is the duality of $H^1(\Rd)$ and the space of function{s}
 of bounded mean oscillation, $BMO(\Rd)$, see \cite{Fefferman_Stein}. For Hardy and $BMO$ spaces  associated with operators
 such duality was  investigated and established in \cite{Duong_Yan}. In the setting which we consider our approach allows us
 to describe {a} natural interpretation of such duality.

Recall that in {the} classical theory the $BMO(\Rd)$ space is defined by the norm
$$\norm{f}_{BMO} = \sup_B |B|^{-1} \int_B |f(x) - f_B| \, dx,$$
where $f_B = |B|^{-1} \int_B f(x)\, dx$ and the supremum is taken over all balls in $\Rd$.
The elements $BMO(\Rd)$ space  are defined up to a constant function. It appears that (in a proper sense) $BMO(\Rd)$ is the
dual of $H^1(\Rd)$. The new cancellation condition \eqref{can} suggests that if $h$ is the $L$ harmonic function then the
 $BMO$ norm associated to $L$ should   be defined
based on  the following expression
$$
\sup_B \inf_c \eee{\mu_h(B)^{-1} \int_B \abs{g(x) - c\, h(x)}^2 h(x) d\mu(x) }^{1/2} <\8.
$$
Note that if $h$ is a constant function then  in
virtue of the John-Nirenberg Inequality the above integral defines the norm
equivalent with the classical $BMO$ definition. It is convenient for us to define $BMO$ in terms of the $L^2$ condition,
see the proof of Lemma \ref{lem_dual}. Theorem \ref{thmB} stated below confirms that the above definition gives a coherent
 description of the  duality between Hardy and $BMO$ spaces associated to the operator $L$ in the considered setting.

The aim of this paper is to study Hardy spaces and their duals for self-adjoint operators defined on spaces of homogeneous type.
In particular, we shall study operators related to some harmonic functions in the sense that the heat semigroup kernel, after
the Doob transform, satisfies lower and upper Gaussian estimates.
Our assumption involving the Doob transform  are specific and the resulting theory is not as general as
in \cite{Hofmann_Memoirs}, but it still includes several interesting applications. For example Laplace operators
with the Dirichlet boundary conditions which were considered by Auscher, Russ, Chang, Krantz and Stein
in
\cite{Auscher_Russ, Chang_Krantz_Stein} can be investigated using the proposed
framework, see Subsection \ref{ssec61} below. Examples also  include  Schr\"odinger operators with certain potentials,  and Bessel operators.
Our result gives a~natural and explicit atomic description of
Hardy spaces with {atoms related} to the $L$-harmonic function $h$. Let us mention that in several examples it is
possible that there exist two or more different  bounded harmonic
functions, see for example \cite{Carron_Coulhon_Hassell}. We hope it is possible to obtain similar description of the
corresponding Hardy and $BMO$ spaces in the case of several harmonic functions but we intend to investigate {such a}
possibility in a different project.

Our characterization, see Theorems \ref{main2} and \ref{thmB} below, is different {from} the ones studied before,
even for well-known classical operators, see e.g. \cite{Auscher_Russ, Chang_Krantz_Stein, DZ_Studia_DK}. {In these papers,
the atoms that describe Hardy spaces can be divided into two classes: some of them are similar to classical atoms, and some of
 them do not satisfy cancellation condition (one can think that a function $|B|^{-1} \chi_B(x)$ is an atom for a proper choice
 of a ball $B$, c.f. \cite{Goldberg_Duke}). Our result gives more homogeneous (and maybe even more natural) description - all
 the atoms satisfy cancellation condition, but with respect to the harmonic function $h(x)$. } Nevertheless the both descriptions
 are equivalent, see Section \ref{appendix} below. A secondary goal of our study is to give a list of examples that satisfy
 assumptions of Theorems \ref{main2} and \ref{thmB}, see Section \ref{sec_examples}. However, we believe that there are many
 more operators that fit to our context.

\subsection{Assumptions and main results.}

Let  $(X,d,\mu)$ be a metric measure space
endowed with a distance $d$
and a nonnegative Borel doubling
measure $\mu$ on $X$, c.f. \cite{Coifman-Weiss_Analyse, CoifmanWeiss_BullAMS}.
Recall that a measure $\mu$ satisfies the  doubling condition provided that there exists a~constant $C>0$ such that for all $x\in X$ and for all $r>0$,
    \begin{eqnarray*}
    \mu(B(x,2r))\leq C\mu(B(x, r)).
    \label{e1.1}
    \end{eqnarray*}
 Note that the doubling property  implies the following
strong homogeneity property,
    \begin{equation}
    \mu(B(x, \lambda r))\leq C\lambda^n \mu(B(x, r))
    \label{e1.2d}
    \end{equation}
for some $C, n>0$ uniformly for all $\lambda\geq 1, r>0$, and $x\in X$.
In Euclidean space with {the} Lebesgue measure, the parameter $n$ corresponds to
the dimension of the space, but in our more abstract setting, the optimal $n$
 need not even to be an integer.

Throughout the   paper we assume that $\mu(X) = \8$. We shall consider operators~$L$, that are always assumed to be self-adjoint,
 non-negative, and defined on a domain $\mathrm{Dom}(L) \subseteq L^2(\mu)$. Moreover, we assume that the semigroup $T_t = \exp(-tL)$
  generated by $L$ has a nonnegative integral kernel
    $$
    T_t f(x) = \int_ X T_t(x,y) f(y) \, d\mu(y)
    $$
that satisfies the pointwise upper Gaussian estimates, i.e. there exist $c,C>0$, such that
    \eq{\label{UG}\tag{UG}
    0\leq  T_t(x,y) \leq C \mu(B(x,\st))^{-1} \exp\eee{-\frac{d(x,y)^2}{ct}}, \quad x,y\in X, \ t>0.
    }

There are several equivalent definitions of  Hardy spaces   $H^1_{L }(X)$ associated with~$L$, see Section \ref{ssecH1} below.
The simplest and most direct is in terms of the maximal operator  associated with the heat semigroup generated by $L$, namely
    \begin{eqnarray}\label{e1.2m}
    M_{L} f(x)=\sup_{t>0}\big|e^{-tL}f(x) \big|
    \end{eqnarray}
with $x\in X, f\in L^2(\mu)$.
Then we define the space $H^1_{L }(X)$  as the completion of the set $\{ f \in L^2(\mu):  \|M_{L} f\|_{L^1(\mu)}< \infty \}$
with respect to $L^1$-norm of the maximal function,
  \begin{eqnarray*}
  \norm{f}_{H^1_L(X)}  = \|M_{L} f\|_{L^1(\mu)}.
  \end{eqnarray*}

\subsubsection{Motivation: an atomic decomposition result.}\label{ssec12}
Let us now recall some results from \cite{Dziubanski_Preisner_Annali_2017}. Assume that we have a space $(X,d,\nu)$ and an
operator $\ll$ related to a semigroup $\tt_t = \exp(-t\ll)$. Notice, that we have changed the notation: $(X, d, \mu)$, $L$,
$T_t$ are replaced by $(X,d,\nu)$, $\ll$, $\tt_t$ (in what follows, the latter will be used for the operators after applying the Doob transform).
Following \cite{Dziubanski_Preisner_Annali_2017}, suppose that the semigroup kernel $\tt_t(x,y)$ satisfies lower and upper Gaussian estimates, i.e.
there exist $c_1, c_2,C>0$ such that
\eq{\label{LG}\tag{ULG}
 C^{-1} \nu(B(x,\st))^{-1} \exp\eee{-\frac{d(x,y)^2}{c_1t}} \leq  \tt_t(x,y) \leq C \nu(B(x,\st))^{-1} \exp\eee{-\frac{d(x,y)^2}{c_2t}}
}
for  $x,y\in X$ and $t>0$.
\prop{prop_existence}{\cite[Proposition 3]{Dziubanski_Preisner_Annali_2017}
Assume that a semigroup $\tt_t$ satisfies \eqref{LG}. Then there exists a function $\w$ such that
\eq{\label{double-bounded-h}
0<c\leq \w(x) \leq C
}
and $\w$ is $\ll$-harmonic in the sense that for all $t>0$,
\eq{\label{L-harm1}
\tt_t\w (x) =\w(x), \quad \text{a.e. } x\in X.
}
}
For details we refer the reader to \cite[Sec. 2]{Dziubanski_Preisner_Annali_2017}. Let us notice that $\varphi$ is
unique (up to a constant), see Corollary \ref{uniqueness}. By Liouville's theorem, the constant functions are the only
bounded harmonic functions when $\ll$ is the the Laplace operator $\Delta$ on ${\mathbb R}^n$. Following \cite{Dziubanski_Preisner_Annali_2017}, we
call a function $a$ a $(\nu,\w)$-atom if there exists a ball $B$ such that:
\eq{\label{w-atoms}
\supp \, a \subseteq B, \qquad \norm{a}_\8 \leq \nu(B)^{-1}, \qquad \int a(x) \w(x) d\nu(x) = 0.
}
The atomic Hardy space      { $H^1_{at} (\nu,\w)$   }is defined then in a standard way using $(\nu,\w)$-atoms.
It is shown in \cite[Theorem 1]{Dziubanski_Preisner_Annali_2017} that  if $\ll$ satisfies \eqref{LG} and an additional geometric
continuity assumption, see \cite[Theorem 1]{Dziubanski_Preisner_Annali_2017},
then for $\varphi$ from Proposition \ref{prop_existence} we have
\begin{eqnarray}\label{e1.6}
\norm{f}_{H^1_{\ll }(X)} \simeq  \norm{f}_{H^1_{at} (\nu,\, \w)}.
\end{eqnarray}

Obviously, the assumption \eqref{LG} is quite restrictive. However, there is a more general version \eqref{ULG'} that includes a
harmonic function $h(x)$, which can have bounded values but not separated from zero, or can be even unbounded. Such harmonic
functions appear e.g. when studying the Dirichlet Laplacian on a domain above the graph of a bounded $C^{1,1}$ function on $\Rd$ or
the exterior of a $C^{1,1}$ compact convex domain in~$\Rd$. Moreover, the same story appears when studying some Schr\"odinger
operators (e.g. $-\Delta + \gamma |x|^{-2}$ on $\Rd$, $n\geq 3$, with $\gamma >0$), or for some Bessel operator defined on a
 weighted half-line. We shall discuss the details in Section \ref{sec_examples}.

\subsubsection{Main results}
The following assumptions are motivated by the notion of the Doob transform (or $h$-transform), see e.g.
 \cite{Gyrya-Saloff-Coste,  Grigoryan_HeatKernels}. Assume that there exists a function $h: \, X \to (0,\8)$ such that:

\begin{enumerate}[label={(H\arabic*)}]

\item \label{H3} : $h$ is $L$-harmonic in the sense that for all $t>0$
\eqx{\label{L-harm}
T_t h (x) =h(x), \quad \text{a.e. } x\in X.
}

\item : The metric-measure space $(X,d,\mu_{h^2})$ is doubling, where $\mu_{h^2}$ is the measure with the density $ h^2(x) d\mu(x)$.
\label{H1}

\item \label{H4} :
 There exist $c_1, c_2, C>0$ such that for $x,y\in X$ and $t>0$ we have
 \begin{equation}
 \label{ULG'}\tag{ULG\textsubscript{h}} \frac{C^{-1}}{\mu_{h^2}(B(x,\st))} \exp\eee{-\frac{d(x,y)^2}{c_1t}}\leq \frac{T_t(x,y)}{h(x)h(y)}
 \leq \frac{C}{\mu_{h^2}(B(x,\st))} \exp\eee{-\frac{d(x,y)^2}{c_2t}}.
 \end{equation}
\end{enumerate}
\smallskip

Let us notice that \ref{H1} and \ref{H4} imply that the action of $T_t$ on $h$ is well defined, even if $h$ is unbounded.
Moreover, Proposition \ref{rem_rem} below says that, in some sense, the assumption \ref{H3} is always true after some mild change of function
$h$. However, we decided to state \ref{H3} as an assumption to emphasize the relation of $L$-harmonicity of $h$ with the estimates \eqref{ULG'}.

Now we define our atomic Hardy space that
will be used to describe $H_{L}^1(X)$.

\begin{Defn}\label{atoms-h}
We call a function $a$ an $[\mu , h]$-atom if there exists a ball $B$ such that:
\al{
\label{at1c}
&\circ \quad  \supp \, a \subseteq B,\\
\label{at1a}
&\circ \quad  \norm{a}_{L^2\eee{h^{-1}\mu}} \leq \mu_{h}(B)^{-1/2},\\
\label{at1b}
&\circ \quad \int a(x) h(x) d\mu(x) = 0.
}
Then, by definition, a function $f$ belongs to the atomic Hardy space $H^1_{at}[\mu , h]$ if
    $
    f = \sum_k \la_k a_k,
    $
where $a_k$ are  $[\mu , h]$-atoms and $\sum_k|\la_k| <\8$. Moreover, define
$$\norm{f}_{{H^1_{at}[\mu , h]}} = \inf \sum_k|\la_k|,$$
where the infimum is taken over all representations of $f$ as above.
\end{Defn}

Observe that by \eqref{at1c}--\eqref{at1a} every  $[\mu,h]-$atom $a$  satisfies the estimate
$$\norm{a}_{L^1(\mu)} \leq \norm{a}_{L^2\eee{h^{-1}\mu}} \mu_h(B)^{1/2} \leq 1,$$
so the series {$f = \sum_k \la_k a_k$} above {converges} in $L^1(\mu)$-norm and a.e. By a standard argument, $H^1_{at}[\mu , h]$ is a Banach space.

The main goal of this paper is to provide a natural and simple atomic {description} of $H^1_L(X)$
 (in the spirit of \eqref{w-atoms}--\eqref{e1.6}). Recall that $A_p(\mu)$ is the Muckenhoupt class, see \eqref{weight} below.
 Our result can be stated in a following way.

\settheoremtag{A}
\begin{theorem}\label{main2}
Suppose that an operator $L$, its semigroup $T_t=\exp(-tL)$, and a function $h(x)$ satisfy the assumptions \ref{H3}--\ref{H4}. There
exists $p_0\in [1, 2]$ such that if $h^{-1}\in A_{p_0}(\mu_{h^2})$, then the spaces $H^1_{L }(X)$ and {$H^1_{at}[\mu , h]$} coincide and
$$\norm{f}_{H^1_{L }(X)} \simeq \norm{f}_{{H^1_{at}[\mu , h]}}.$$
\end{theorem}

We would like to emphasize  that in Theorem \ref{main2} the semigroup $T_t$ does not need to satisfy \eqref{LG}. Hence \eqref{double-bounded-h}
is not necessarily valid so it can happen that $\inf h(x) = 0 $ or $\sup h(x) = \8$ (this is the case in many interesting examples).
Therefore Theorem \ref{main2} can be seen as a generalization of \eqref{e1.6} from \cite{Dziubanski_Preisner_Annali_2017}. However,
there is a small cost here, namely we change $L^\8$-type condition on size of atoms into weighted $L^2$-type condition.

Also, note that in \cite[Theorem 1]{Dziubanski_Preisner_Annali_2017} the result requires the following geometric assumption: for
every $x \in X$ the function $r \mapsto \nu(B(x,r))$ is a bijection on $(0,\8)$. Our approach does not require this condition.

Our proof of Theorem \ref{main2} uses strongly the Doob transform, see Subsection \ref{ssec_Doob}. More precisely, we can introduce
a new semigroup $\tt_t = \exp(-t\ll)$ by \eqref{Doob_semi} which acts on $(X,d,\nu)$, $d\nu(x) = h^2(x) d\mu(x)$ and satisfies \eqref{LG}
on this changed metric-measure space. Moreover,
$$H^1_{L }(X) \ni f \mapsto h^{-1}f \in H^1_{\ll,   h^{-1}}(X)$$
is an isometry between $H^1_{L}(X)$ (related to the measure $\mu$) and a weighted Hardy space $H^1_{\ll,   h^{-1}}(X)$ (related to the measure $\nu$).
 Therefore, we shall study weighted Hardy spaces for operators $\ll$ satisfying \eqref{LG} in Section \ref{sec_proof1} below.
  The proof of \eqref{e1.6} in \cite{Dziubanski_Preisner_Annali_2017} uses different methods to the ones used here.
  In \cite{Dziubanski_Preisner_Annali_2017} the key step is to use a theorem of Uchiyama \cite{Uchiyama}, which relies
  on the analysis of grand maximal function. Our approach is based on atomic decompositions for weighted tent spaces, see \cite{Russ,Song_Wu, Liu_Song}.

The second goal of this paper is to study the space of functions of bounded mean oscillation,
the dual of $H^1_{at}[\mu, h]$. By definition, a function $g$ is in $BMO[\mu,h] $ if
\eqx{
\norm{g}_{BMO[\mu,h]} := \sup_B \inf_c \eee{\mu_h(B)^{-1} \int_B \abs{g(x) - c\, h(x)}^2 h(x) d\mu(x) }^{1/2} <\8.
}
In a standard way, elements of $BMO[\mu,h] $ are classes $\set{{g+c\, h} \ : \ c\in \CC}$.
A natural analogue of the Fefferman-Stein duality result \cite{Fefferman_Stein} is the following:

 \settheoremtag{B}
\begin{theorem}\label{thmB}
Suppose that an operator $L$, its semigroup $T_t=\exp(-tL)$, and a function $h(x)$ satisfy the assumptions \ref{H3}--\ref{H4}.
 There exists $p_0\in (1,2]$ such that if $h^{-1}\in A_{p_0}(\mu_{h^2})$, then
$BMO[\mu,h]$ is the dual to the Hardy space $H^1_{at}[\mu, h]$.
\end{theorem}

The proof of Theorem \ref{thmB} and further details are discussed in Section \ref{sec_BMO}. The outline of the remainder of
 the paper is as follows. In Section  \ref{sec_prel} we recall some known facts on: Hardy spaces, the Doob transform,
 Gaussian estimates, and prove some preliminary results. In Section \ref{sec_proof1} we study the weighted Hardy spaces
 and the corresponding atomic decompositions. In Sections   \ref{sec_proof3} and \ref{sec_BMO} we prove our main results,
 Theorems  \ref{main2} and \ref{thmB}, respectively. In Section \ref{sec_examples} we provide several examples of operators that satisfy our assumptions.

\section{Preliminaries.}
\label{sec_prel}
\setcounter{equation}{0}

We now set notation and some common concepts that will be used throughout the course of the proofs. Let $(X,d,\mu)$ be
 a metric measure space endowed with a distance $d$ and a nonnegative Borel doubling measure $\mu$ on $X$.
 The operator $L$ is related to the semigroup $T_t$ on the space $(X,d,\mu)$, whereas $\ll$ is related to $\tt_t$ on $(X,d,\nu)$.
 The difference is that we always assume that $\ll$ satisfies \eqref{LG}, whereas $L$ satisfies more general condition \eqref{ULG'}.
 As a consequence $\w(x)$ from Proposition \ref{prop_existence} is the harmonic function
 for $\ll$ that is bounded from above and from below. However, the harmonic function $h$ related
 to $L$ is in general unbounded (either from above or from below). Finally, the letters $c,C$ are positive
 constants that may change from line to line. The notation $A\simeq B$ means that $C^{-1}A \leq B \leq CA$.

\subsection{Hardy spaces $H^1_L(X)$.}
\label{ssecH1}	
Let us start with giving a few definitions of the Hardy space $H^1_L(X)$ adapted to an operator $L$. At the end
all these definitions are the same Hardy space that we shall denote $H^1_L(X)$. In Subsection \ref{ssec11}
we already defined $H^1_L(X)=H^1_{L, \max}(X)$ by means of the
maximal function. Let us also recall
  the
following Lusin (area) function $S_{L}f$ and Littlewood-Paley function $G_{L}f$ associated to the
heat semigroup generated by $L$
\begin{equation}
S_{L}f(x):=\left(\iint_{d(x,y)< t}
|t^2Le^{-t^2L} f(y)|^2 \frac{d\mu(y)}{\mu(B(x,t))}{\frac{dt}{t}}\right)^{1/2},
\quad x\in X,
\label{e2.7}
\end{equation}
 and
\begin{equation}
G_{L}f(x):=\left(\int_0^{\infty}
|t^2Le^{-t^2L} f(x)|^2  {dt\over t}\right)^{1/2},
\quad x\in X.
\label{e2.8}
\end{equation}
 We define the Hardy space $H^1_{L, S}(X)$
 as the completion of $\{f\in L^2(\mu):  \|S_L f\|_{L^1(\mu)} < \infty \}$  in $L^1(\mu)$, see \cite[Theorem 4.7]{Auscher_McIntosh_Morris}, with respect to
 $L^1$-norm of the Lusin (area) function, i.e.
\begin{equation*}
\|f\|_{H^1_{L, S}(X)}=\|S_L f\|_{L^1(\mu)}.
\label{e2.9}
\end{equation*}
The space  $H^1_{L, G}(X)$ is defined analogously. Now, we shall {discuss} another approach to
 atomic decomposition of $H^1_L(X)$, which {works} in a more general context, but gives different
 (and in some sense more complicated) atoms. At this moment it is enough to make only assumptions from Subsection \ref{ssec11}.
Following \cite{Hofmann_Memoirs} let us define an {\it $L$-atom} $a$ as follows. Assume that there exists a ball $B=B(y_0,r) \subseteq X$
and a function $b \in \mathrm{Dom}(L)$ such that for $k=0,1$ we have:
\begin{eqnarray}\label{atom}
a=Lb,\qquad \supp \, L^k b \subseteq B, \qquad \norm{(r^2L)^k b }_{L^2(\mu)} \leq r^2 \mu(B)^{-1/2}.
\end{eqnarray}
Using $L$-atoms, one defines an atomic Hardy space $H^1_{L, {\rm {at}}}(X)$ as in \cite[Definition 2.2]{Hofmann_Memoirs}.
 In \cite[Theorem 7.1]{Hofmann_Memoirs} Hofmann et. al. proved that
\eq{\label{dif1}
\norm{f}_{H^1_{L, S}(X)} \simeq \norm{f}_{H^1_{L, {\rm {at}}}(X)} \leq C \norm{f}_{H^1_{L}(X)}.
}
Later, in \cite[Theorem 1.3]{Song_Yan_2018}, a complementary estimate was proved, namely
\eqx{\label{dif2}
\norm{f}_{H^1_{L }(X)} \leq C \norm{f}_{H^1_{L, {\rm at}}(X)}.
}
Moreover, results from {\cite[Theorem 1.2]{Hu}} imply that
\eqx{\label{dif3}
\norm{f}_{H^1_{L, S}(X)} \simeq \norm{f}_{H^1_{L, G}(X)}.
}
Therefore, all the definitions above lead to the same Hardy space that we shall denote
$$
H^1_L(X) := H^1_{L,max}(X) = H^1_{L, S}(X)=H^1_{L, G}(X)=H^1_{L, {\rm at}}(X).
$$
 Let us also mention that $H^1_L(X)$ has also
equivalent norms in terms of non-tangential maximal function and analogues with Poisson semigroup, see \cite{Hofmann_Memoirs, Song_Yan_2016, Song_Yan_2018}.

\begin{Rem} From \cite[Lemma 9.1]{Hofmann_Memoirs} it follows   that if the semigroup $T_t=\exp (-tL)$ related to an operator $L$
is conservative, i.e.,
$$
\int_X T_t(x,y) d\mu(y)=1, \quad t>0, \ x\in X,
$$
then for every $L$-atom  $a$ we have
$
\int_X a(x) d\mu(x)=0.
$
\end{Rem}

\subsection{Doob transform}\label{ssec_Doob}
In this section we describe one of the most important tools for this paper, i.e. the Doob transform (or $h$-transform),
see e.g. \cite{Gyrya-Saloff-Coste, Grigoryan_HeatKernels}. Assume that an operator $L$ related to a metric measure space
$(X,d,\mu)$ and a function {$h$} satisfy \ref{H3}--\ref{H4}. Notice that here we do { not} assume that $h$ is bounded
neither from above nor from below. See Section \ref{sec_examples} for examples.

On $(X,d)$ define a new measure $d\nu(x) = h^2(x) d\mu(x)$ and a new kernel
\eq{\label{Doob_semi}
\tt_t(x,y) = \frac{T_t(x,y)}{h(x)h(y)}.
}
By \ref{H1} the space $(X,d,\nu)$ satisfies the doubling condition. The inequalities from \ref{H4}  for $T_t$ are equivalent to
\eqref{LG} for $\tt_t$. The Doob transform is a simple multiplication operator
\eqx{
f\mapsto h^{-1}f.
}
Observe that
$$\norm{f}_{L^2(\mu)} = \norm{h^{-1}f}_{L^2(\nu)}$$
so the Doob transform is an isometry between these two $L^2$ spaces. Moreover, a simple calculation shows that $\tt_t$ is a semigroup
and its generator $\ll$ is also self-adjoint (as an image of $L$ under isometry). However, the Doob transform is not an isometry between
$L^1$ spaces but we still have $\norm{f}_{L^1(\mu)} = \norm{h^{-1}f}_{L^1(h^{-1} \nu)}.$

Recall now that $H^1_L(X)$ corresponds to the measure $\mu$, whereas $H^1_{\ll}(X)$ is defined with respect to $\nu=\mu_{h^2}$.
 A crucial observation in this paper
is  the following  proposition, where $H^1_{\ll, G, h^{-1}}(X)$ and $H^1_{\ll, max,h^{-1}}(X)$
 are {  weighted} Hardy spaces that we define in Section \ref{sec_proof1} below.
\prop{prop_Doob}{
Let $f\in L^1(\mu)$. Then
\alx{
\norm{f}_{H^1_{L, G}(X)} &= \norm{f}_{H^1_{\ll, G, h^{-1}}(X)},\\
\norm{f}_{H^1_{L, max}(X)} &= \norm{f}_{H^1_{\ll, max, h^{-1}}(X)}.\\
}
}
\begin{proof}
It is enough to notice that
\spx{
\int_X \eee{\int_0^\8 \abs{t^2LT_t f(x)}^2}^{1/2}\, d\mu(x) =\int_X \eee{\int_0^\8 \abs{t^2\ll\tt_t \eee{h^{-1}f}(x)}^2}^{1/2}\,  \frac{d\nu(x)}{h(x)}
}
and
\spx{
\int_X \sup_{t>0} \abs{\int_X T_t(x,y) f(y) d\mu(y)} d\mu(x)
=\int_X \sup_{t>0} \abs{\int_X \tt_t(x,y) \frac{f(y)}{h(y)} d\nu(y)} \frac{d\nu(x)}{h(x)}.
}
\end{proof}

The next statement essentially says that,
for the purpose  of our discussion here, assumption \ref{H3} is automatically fulfilled provided that  assumptions \ref{H1} and \ref{H4} are valid.

\begin{Prop}\label{rem_rem}
Assume that for a semigroup $T_t$ there exists a function $\wt{h}$ such that \ref{H1} and \ref{H4} are satisfied. Then, there exist
$C>0$ and a function $\w : X \to \RR$ such that $C^{-1} \leq \w(x) \leq C$  and for all $t>0$ we have
$$T_t(\w \wt{h}) (x) = \w \wt{h}(x), \quad \text{a.e. }x\in X.$$
\end{Prop}

\begin{proof}
 Assume that $\wt{h}$ is such that \ref{H1} and \ref{H4} hold. Then, after the Doob transform the
 semigroup $\tt_t$ satisfies \eqref{LG} and we obtain $\w$ satisfying \eqref{double-bounded-h} and \eqref{L-harm1}, see
 Subsection \ref{ssec12}. In particular, for $t>0$,
 $$T_t(\w \wt{h})(x) =\wt{h}(x) \tt_t \w(x) = \wt{h}(x) \w(x).$$
 Thus, $h= \wt{h}\w$ is $L$-harmonic and still satisfies \ref{H1}--\ref{H4}.
\end{proof}

\subsection{Semigroups with two-sided Gaussian bounds.}\label{ssec23}
In this subsection we assume that $\tt_t(x,y)$ is a semigroup that satisfy \eqref{LG} on the space $(X,d,\nu)$. Then,
there exists a function $\w$, such that \eqref{double-bounded-h} and \eqref{L-harm1} are satisfied,
see {Proposition \ref{double-bounded-h} and} \cite[Sec. 2]{Dziubanski_Preisner_Annali_2017}.
It is well known that \eqref{LG} implies certain H\"older regularity
in the space variable for $T_t(x,y)$. For a~simple proof see \cite[Sec. 4]{Dziubanski_Preisner_Annali_2017}.

\prop{Holder}{\cite[Corollary 14]{Dziubanski_Preisner_Annali_2017}
Assume that the semigroup kernel $\tt_t(x,y)$ satisfies $\eqref{LG}$ and $\w$ is the related $\ll$-harmonic function.
 There exist $C,c,\delta>0$ such that
\eq{\label{Lips}
\abs{\frac{\tt_t(x,y)}{\w(x)\w(y)} - \frac{\tt_t(x,y_0)}{\w(x)\w(y_0)}} \leq C \eee{\frac{d(y,y_0)}{\st}}^\de \nu(B(x,\st))^{-1} \exp\eee{-\frac{d(x,y)^2}{ct}}
}
whenever $d(y,y_0)<\sqrt{t}$.
}
Let us remark that $\w(x) \simeq C$, so we could skip $\w(x)$ in \eqref{Lips}. However, we need to divide
by $\w(y)$ and $\w(y_0)$ to get H\"older-type inequality. Let us notice that Proposition~\ref{Holder} implies the following corollary.

\cor{uniqueness}{
If $\tt_t(x,y)$ satisfies \eqref{LG}, then $\w$ is (up to a constant) the unique bounded harmonic function.
}
\begin{proof}
Let $\wt{\w}$ be such that $\tt_t \wt{\w}(x) = \wt{\w}(x)$ and $|\wt{\w}(x)|\leq C$. From \eqref{Lips} for $\st>d(y_1,y_2)$ we have
\spx{
\abs{\frac{\wt{\w}(y_1)}{\w(y_1)} -\frac{\wt{\w}(y_2)}{\w(y_2)}}&=\abs{\frac{\tt_t\wt{\w}(y_1)}{\w(y_1)} -\frac{\tt_t\wt{\w}(y_2)}{\w(y_2)}}  \\
&\leq \int_X \abs{\frac{\tt_t (y_1,x)}{\w(y_1)} - \frac{\tt_t (y_2,x)}{\w(y_2)}} \wt{\w}(x) \, d\nu(x)\\
&\leq C \norm{\wt{\w}}_\8 \eee{\frac{d(y_1,y_2)}{\st}}^\de  \int_X \nu(B(x,\st))^{-1} \exp\eee{-\frac{d(x,y)^2}{ct}} \, d\nu(x)\\
&\leq C \norm{\wt{\w}}_\8 \eee{\frac{d(y_1,y_2))}{\st}}^\de.
}
Taking $t\to \8$ we arrive at $\wt{\w}/\w \equiv C$. This completes the proof of Corollary~\ref{uniqueness}.
\end{proof}

Let us state another consequence of Proposition \ref{Holder} that we shall use in Section \ref{sec_proof1}.

\prop{Lips2}{
Assume that $\tt_t(x,y)$ and $\delta,c>0$ are as in Proposition~\ref{Holder} and that $\kk_t(x,y)$ is the kernel
of the operator $ t\ll \exp(-t\ll)$. For $d(y,y_0)<\sqrt{t}$ we have
\eq{\label{Lips3}
\abs{\frac{\kk_t(x,y)}{\w(x)\w(y)} - \frac{\kk_t(x,y_0)}{\w(x)\w(y_0)}} \leq
C \eee{\frac{d(y,y_0)}{\st}}^\de \nu(B(x,\st))^{-1} \exp\eee{-\frac{d(x,y)^2}{2ct}}.
}
}
\pr{
By the self-improvement property of Gaussian estimates we have that
\eq{\label{self-imp}
\abs{\kk_t(x,z)} \leq C \nu(B(x,\st))^{-1} \exp\eee{-\frac{d(x,z)^2}{ct}},
}
see e.g. {\cite{Grigoryan_HeatOld} or \cite[Theorem 4.]{S2}}. Observe that
$$\kk_t(x,y) = 2 \int_X \kk_{t/2}(x,z) \tt_{t/2}(z,y)\, d\nu(z).$$
Next, by \eqref{Lips},
\spx{
&\hspace{-0.2cm}\abs{\frac{\kk_t(x,y)}{\w(x)\w(y)} - \frac{\kk_t(x,y_0)}{\w(x)\w(y_0)}} \\
&\leq C \int_X \abs{\kk_{t/2}(x,z)}
 \abs{\frac{\tt_{t/2}(z,y)}{\w(y)} - \frac{\tt_{t/2}(z,y_0)}{\w(y_0)}} \, d\nu(z)\\
&\leq C \eee{\frac{d(y,y_0)}{\st}}^\de \nu(B(x,\st))^{-1} \int_X \nu(B(z,\sqrt{t}))^{-1} \exp\eee{-\frac{d(x,z)^2
 + d(z,y)^2}{ct}} \, d\nu(z)\\
&\leq C \eee{\frac{d(y,y_0)}{\st}}^\de \nu(B(x,\st))^{-1} \exp\eee{-\frac{d(x,y)^2}{4ct}} \int_X \nu(B(z,\sqrt{t}))^{-1}
 e^{-\frac{d(z,y)^2}{2ct}}\, d\nu(z)\\
&\leq  C \eee{\frac{d(y,y_0)}{\st}}^\de \nu(B(x,\st))^{-1} \exp\eee{-\frac{d(x,y)^2}{4ct}}.
}
The proof of Proposition~\ref{Lips2}  is complete.
}

\section{Weighted Hardy spaces}
\label{sec_proof1}
\setcounter{equation}{0}

The theory of weighted Hardy spaces in $\Rd$ was studied in \cite{Garcia-Cuerva, Stromberg}. In the more general
context of spaces of homogeneous type the reader is referred to \cite{Yang_Yang, Liu_Song, Hu, Song_Yan_2016, Song_Yan_2018}
and references therein.

 \subsection{Muckenhoupt weights}
Recall that a  non-negative function  $w$  defined on $X$ is called a weight if it is   locally integrable.
We  denote by $\mu_w(A) = \int_A w(x) d\mu(x)$ the weighted measure, and
 by $\norm{f}_{L^p_w(\mu)} = (\int_X |f(x)|^p w(x)\, d\mu(x))^{1/p}$ the weighted $L^p$-norm.

We say that $w$ is in the Muckenhoupt class $A_p(\mu)$, $p>1$,  if   there is a constant $C$ such that
\begin{eqnarray}\label{weight} \hspace{2.8cm}
\left({1\over \mu(B)} \int_B w(x) d\mu(x)  \right)    \left({1\over \mu(B)} \int_B w(x)^{-1/(p-1)}(x)d\mu(x)\right)^{p-1} \leq C
\end{eqnarray}
holds for every ball $B\subset X$.
The class $A_1$ is defined replacing \eqref{weight} by
 \begin{eqnarray*}
 \|w^{-1}\chi_B\|_{\infty} \left({1\over \mu(B)} \int_B w(x) d\mu(x)  \right)  \leq C,
 \end{eqnarray*}
 where $\chi_B$ is the characterization function of the ball $B$. The class $A_{\infty}(\mu)$ is defined as the union of the $A_p(\mu)$ classes for
 $1\leq p<\infty,$ i.e.,
   $A_\8(\mu) = \bigcup_{p\geq 1}A_p(\mu)$.
In the sequel,   we shall use the following standard properties of $A_p(\mu)$ weights. For details we refer the reader to \cite{Garcia-Cuerva, Stein, Stromberg}.
\begin{Lem} \label{le2.1}
(i) If $p>1$ and $w\in A_p(\mu)$, then there exists $\e>0$ such that $w\in A_{p-\e}(\mu)$.

 (ii)
Assume that $p\geq 1$, $w\in A_p(\mu)$. There exists $C>0$ such that {for a ball $B$ and a measurable set} $E\subseteq B$ we have
 $$
\left( \frac{\mu(E)}{\mu(B)} \right) ^p \leq C  \frac{\mu_w(E)}{\mu_w(B)}.
 $$
 \end{Lem}

\subsection{Weighted Hardy spaces}\label{ssec32}
In this section the weight $w$ belongs to $A_\8(\mu)$ and $L,T_t$ are as in Subsection \ref{ssec11}.
Recall that
the
Lusin (area) function $S_{L}f$ and Littlewood-Paley function $G_{L}f$
  are given by \eqref{e2.7} and \eqref{e2.8}, respectively.
We define $H^1_{L, S, w} (X)$ as the completion of the set
$\{ f\in L^2(\mu): S_L f\in L^1_w(\mu)\}$ {in $L^1_w(\mu)$}, with respect to the norm
$$
\|f\|_{H^1_{L,  S, w} (X)}  =\|S_L f\|_{L^1_w(\mu)};
$$
The space $H^1_{L, G, w}(X)$ are defined analogously.
There are several results on the weighted Hardy {spaces} $H^1_{L, w}(X)$.
In \cite[Theorem 1.2]{Hu} it is proved that for $w\in A_\8(\mu)$,
\begin{eqnarray}\label{e3.8}
H^1_{L, S, w}(X) = H^1_{L, G, w}(X).
\end{eqnarray}

In \cite{Liu_Song, Song_Wu}  the authors proved a weighted version of \eqref{dif1}. 
Suppose that $M\in {\mathbb N}$ and $
w\in A_p,  {1  \leq p \leq  2}$, we say that a function $a\in L^2(\mu)$
 is called an $(L,  M, w)$-{\it atom}
 if there exists a ball $B = B(y_0,r)$ in $X$ and a function $b$ such that: $b\in \mathrm{Dom}(L^M)$ and for $k=0, 1, \dots, M$ we have
$$
\quad a=L^Mb, \quad \supp \, L^k b \subseteq B,\quad \norm{(r^2L)^k b }_{L^2_w(\mu)} \leq r^{2M} \mu_w(B)^{-1/2},
$$
c.f. \eqref{atom} for non-weighted $L$-atoms.

\defn{atoms-Lw}{  Suppose that $M\in {\mathbb N} $ and
$w\in A_p (\mu), {1  \leq p\leq 2}$.
A function $f$ belongs to $\HH^1_{L, at, M, w}(X)$ if
$
f = \sum_k \la_k a_k,
$
where $a_k$ are $(L,M, w)$-atoms, $\sum_k|\la_k| <\8$ and the series converges in $L^2(\mu)$. Define
$$\norm{f}_{H^1_{L, at, M, w}(X)} =\inf \sum_k |\la_k|,$$
where $f\in \HH^1_{L, at, M, w}(X)$ and $f$ is decomposed as above. Then $H^1_{L, at, M, w}(X)$ is defined as a completion
of $\HH^1_{L, at, M, w}(X)$ in the norm $\norm{\cdot}_{H^1_{L, at, M, w}(X)}$.
}

{Note that in Definition~\ref{atoms-Lw}, we define  $\HH^1_{L, at, M, w}(X)$ to be the normed space obtained by $L^2(\mu)$ convergence.
		This approach to the definition of adapted $H^1$ space was used in
		\cite{Hofmann_Memoirs, Hofmann_Mayboroda, Hofmann_Mayboroda_McIntosh, Liu_Song, Song_Wu}.
		For further discussion see \cite[p. 879 and Rem. 3.15]{Auscher_McIntosh_Morris}, \cite[Def. 3.4 and Theorem 3.5]{Hofmann_Mayboroda_McIntosh}, {and Section \ref{appendix2}}.
}
The following result was proved in \cite{Liu_Song} in the case $X=\Rd$ and in {\cite[Theorem 1.10]{Song_Wu}} when $X$ is a~space of homogeneous type.
\begin{Thm}\label{Liu_Song}
Assume that $(X,d,\mu)$ is a doubling metric-measure space and $T_t = \exp(-tL)$ is a semigroup satisfying \eqref{UG}.
Assume that $
w\in A_p(\mu)$, $1\leq p \leq 2$, and $M\in {\mathbb N}, M>(p-1)n/2$, where $n$ is as in \eqref{e1.2d}. Then
$$\norm{f}_{H^1_{L, S, w}(X)}  \simeq \norm{f}_{ H^1_{L, at, M, w}(X)}.$$
\end{Thm}

Consequently, one may write $H^1_{L, at,   w}(X)$ in place of $H^1_{L, at, M, w}(X)$ when $
w\in A_p(\mu)$, $1\leq p\leq 2$, and $ M>(p-1)n/2$ as these spaces are all equivalent.
Having in mind \eqref{e3.8} and Theorem~\ref{Liu_Song}  we write
$$
H^1_{L, w}(X):=H^1_{L, S, w}(X)=H^1_{L, G, w}(X)=H^1_{L, at,   w}(X):=H^1_{L, at, M,  w}(X)
$$
for   $M>(p-1)n/2$.

Next, for an operator $L$ related to a metric measure space
 $(X,d,\mu)$ and a function $h(x)$ satisfy \ref{H3}--\ref{H4}, and we consider the semigroup $\tt_t$
 corresponding to the measure  $d\nu(x) = h^2(x) d\mu(x)$,  as in Subsection \ref{ssec_Doob}.
 By \ref{H1} the space $(X,d,\nu)$ satisfies the doubling condition. The inequalities \eqref{ULG'}
 for $T_t$ are equivalent to \eqref{LG} for $\tt_t$. Recall that,  $\tt_t$ is a semigroup and its
 generator $\ll$ is also self-adjoint, see Section~\ref{ssec_Doob}. As~in the above notation corresponding
 to the operator $L$ the  spaces $H^1_{\ll, S, w}(X)$, $H^1_{\ll, G, w}(X) $, and $H^1_{\ll, at, w}(X)$ related to $\ll$  are defined analogously
   and all these weighted Hardy  spaces coincide, i.e.
$$
H^1_{\ll, w}(X) = H^1_{\ll, S, w}(X)=H^1_{\ll, G, w}(X)=H^1_{\ll, at,     w}(X):=H^1_{\ll, at, M,  w}(X)
$$
when
$
w\in A_p, 1< p\leq 2$ and $M\in {\mathbb N}, M>(p-1)n/2$.

\subsection{Alternative  atomic characterization with cancellation  condition}
We shall prove atomic decompositions for $H^1_{\ll, w}(X)$
with natural and simple atoms related to $\nu$, $w$, and $\ll$-harmonic function $\w$. Recall that the existence of
$\w$ satisfying \eqref{double-bounded-h} follows from \eqref{LG}, see \cite[Sec. 2]{Dziubanski_Preisner_Annali_2017}.
\defn{atoms-hh}{
We call a function $a$ a $[\nu,\w,w]$-atom if there exists a ball $B$ such that:
\alx{
&\circ \quad \supp\, a \subseteq B,\\
&\circ \quad  \norm{a}_{L^2_w(\nu)} \leq \nu_w(B)^{-1/2},\\
&\circ \quad \int_B a(x) \w(x) d\nu(x) = 0.
}
Then, by definition, a function $f$ belongs to the atomic Hardy space $H^1_{at}[ \nu,  \w, w]$ if
$
{f = \sum_k \la_k a_k,}
$
where $a_k$ are $[\nu,\w,w]$-atoms, $\sum_k|\la_k| <\8$, and the series converges in $L^1_w(\nu)$. Moreover, for such representations the expression
$$\norm{f}_{H^1_{at}[ \nu,  \w, w]} = \inf \sum_k|\la_k|$$
defines a norm.
}
Observe that in Definition \ref{atoms-hh} the atoms satisfy $\norm{a}_{L^1_w(\nu)}\leq 1$, so the series $\sum_k \la_k a_k$
converges in $L^1_w(\nu)$-norm and a.e.\,. Moreover, the space $H^1_{at}[ \nu,  \w, w]$ is a Banach space.

The main result of this section is the following theorem. It states atomic characterization of the Hardy space $H^1_{\ll, w}(X)$.
Later, we shall deduce Theorem \ref{main2} from Theorem~\ref{main1} by using Doob's transform.
\thm{main1}{
Assume that $\ll$ satisfies \eqref{LG} and $\w$ is the associated bounded $\ll$-harmonic function, see \eqref{double-bounded-h}
and \eqref{L-harm1}. Let $p_0=(n+\delta)/n$, where $\delta$ is from Proposition \ref{Lips2}
and $n$ is as in \eqref{e1.2d}. If $w$ is a~weight in $A_{p_0}(\nu)$, then
$$H^1_{\ll, w} (X) =H^1_{at}[ \nu,  \w, w].
$$
}

\begin{proof}
Let $\w$ be the harmonic function for $\ll$, $C^{-1} \leq \w(x) \leq C$, see Subsection \ref{ssec23}.

{\bf Proof of $H^1_{\ll, w} (X) \subseteq  H^1_{at}[ \nu,  \w, w]$.}
Assume that $f\in H^1_{\ll, w}(X)=H^1_{\ll, at,   w}(X)$, see Subsection \ref{ssec32}. {For $p_0=1+\delta/n$
we have that $(p_0-1)n/2 = \delta/2 <1$ so we can take $M=1$ in Theorem \ref{Liu_Song}.}
 {We can} assume that $f$ is in a dense subspace {$\HH_{\ll,at, 1, w}(X)$}, so that we have $\la_k$
and {$(\ll, 1, w)$-atoms} $a_k$ as in Definition \ref{atoms-Lw} such that {$f=\sum_k\la_k a_k$}
 (convergence in $L^2(\nu)$ and in $L^1_w(\nu)$ and a.e.) and
$$
\norm{f}_{H^1_{\ll, w}(X)} \simeq \sum_k |\la_k|.
$$
Observe that {$(\ll, 1, w)$-atoms} satisfy localization and size condition of Definition \ref{atoms-hh}, so to prove that $a_k$
are $[\nu,\w,w]$-atoms we only need to show that
\eq{\label{cancel1}
\int_B a(x) \w(x) \, d\nu(x) = 0
}
for $a=a_k$. This will be enough since then $\norm{f}_{H^1_{at}[ \nu,  \w, w]} \leq C \norm{f}_{H^1_{\ll,w}(X)}$ on a dense subset of $H^1_{\ll,w}(X)$.

To prove \eqref{cancel1}, we follow the argument similar to \cite[Lemma 9.1]{Hofmann_Memoirs}. Recall that $\w$ is bounded and $\tt_t$
are uniformly bounded on $L^p(\nu)$, $1\leq p \leq \8$. By the functional calculus we have
$$(I+\ll)^{-1} = \int_0^\8 e^{-t} \tt_t dt$$
and, by \eqref{L-harm1},
\eq{\label{resolvent}
(I+\ll)^{-1} \w(x) = \w(x), \qquad \text{for a.e. } x.
}

Using \eqref{resolvent} twice,
\spx{
\int_B a(x) \w(x) \, d\nu(x) &= \int_B a(x) (I+\ll)^{-1}\w(x) \, d\nu(x)\\
&= \int_B (I+\ll)^{-1}\ll b (x) \w(x) \, d\nu(x)\\
&= \int_B (I+\ll)^{-1}(I+\ll)b (x) \w(x) \, d\nu(x) - \int_B (I+\ll)^{-1}b (x) \w(x) \, d\nu(x)\\
&= \int_B b (x) \w(x) \, d\nu(x) - \int_B b (x) (I+\ll)^{-1} \w(x) \, d\nu(x)\\
& = 0.
}
Let us notice that in the calculations above, we use that $a$ and $b$ have compact supports, $\w$ is bounded, $\tt_t$ has the upper Gaussian estimates, and
\eq{\label{unwei}
\norm{a}_{L^1(\nu)} \leq \norm{a}_{L^2_w(\nu)} \eee{\int_B w^{-1}}^{1/2} \leq \norm{a}_{L^2_w(\nu)} \frac{\nu(B)}{\nu_w(B)^{1/2}}<\8.
}
Here we have also used $A_2(\nu)$ condition for $w$. The same estimate holds for $b$.

{\bf Proof of $H^1_{at}[ \nu,  \w, w] \subseteq  H^1_{\ll, w} (X) $.}
First, let us show that for every  $[\nu,\w,w]$-atom $a$, there is a constant  $C$ independent of $a$ such that
\begin{eqnarray}\label{ggg}
\norm{G_{\ll} a}_{L^1_w(\nu)}\leq C,
\end{eqnarray}
where  $G_{\ll} f(x)= \eee{\int_0^\8 \abs{t^2\ll \tt_tf(x)}^2 \frac{dt}{t}}^{1/2}$.

Let $a$ be a $[\nu,\w,w]$-atom, so that $\supp\, a \subseteq B = B(y_0,r)$. Then, since $w\in A_2(\nu)$ and $G_{\ll}$ is bounded on $L^2_w(\nu)$,
$$\norm{G_{\ll} a}_{L^1_w(2B,\nu)}\leq \nu_w(2B)^{1/2} \norm{G_{\ll} a}_{L^2_w(\nu)} \leq \nu_w(2B)^{1/2} \norm{a}_{L^2_w(\nu)} \leq C.$$

Let $x\not\in 2B$ and $y\in B$. Then, $d(x,y) \simeq d(x,y_0) >r$. Let $\kk_{t^{2}}=t^2\ll \tt_{t^2}$ be as in Proposition \ref{Lips2}.
\spx{
G_{\ll} a(x)^2 &= \int_0^\8 \abs{\int_X \kk_{t^2} (x,y) a(y) \, d\nu(y)}^2  \frac{dt}{t} \\
&= \int_0^r + \int_r^\8 =: E_1+E_2.
}
Observe that
\sp{\label{measures}
\nu(B(x,t))^{-1} &= \frac{\nu(B(x,d(x,y_0)))}{\nu(B(x,t))}  \nu(B(x,d(x,y_0)))^{-1} \\
&\leq C \eee{1+\frac{d(x,y_0)}{t}}^n \nu(B(x,d(x,y_0)))^{-1},
}
where $n>0$ is the doubling dimension, see \eqref{e1.2d}.
In $E_1$ we have $t\leq r < d(x,y_0)$. Let $\delta>0$ be as in \eqref{Lips3}. Using \eqref{self-imp} and \eqref{measures},
 \spx{
 E_1 &\leq C \int_0^r\eee{ \int_B \nu(B(x,t))^{-1} \exp\eee{-\frac{d(x,y)^2}{ct^2}} |a(y)| d\nu(y)}^2  \frac{dt}{t}\\
 &\leq C \nu(B(x, d(x,y_0)))^{-2} \norm{a}^2_{L^1(\nu)} \int_0^r \eee{\frac{d(x,y_0)}{t}}^{2n} \exp\eee{-\frac{d(x,y_0)^2}{ct^2}} \frac{dt}{t}\\
  &\leq C \nu(B(x, d(x,y_0)))^{-2} \norm{a}^2_{L^1(\nu)} \int_0^r \eee{\frac{d(x,y_0)}{t}}^{-{2}\de}  \frac{dt}{t}\\
 &\leq C \frac{r^{{2}\delta}}{d(x,y_0)^{{2}\delta}} \nu(B(y_0, d(x,y_0)))^{-2}\norm{a}^2_{L^1(\nu)} .
 }
Recall that $\int a(x) \varphi(x) \, d\mu(x) =0$. For $E_2$ we note that  $d(y,y_0) < r$, so \eqref{Lips3} and \eqref{measures}
yield
\spx{
E_2 &\leq C \int_r^\8 \abs{ \int_B \eee{\frac{\kk_{t^2}(x,y)}{\w(y)} - \frac{\kk_{t^2}(x,y_0)}{\w(y_0)}}a(y) \w(y) d\nu(y)}^2 \frac{dt}{t}\\
 &\leq C\int_r^\8  \eee{\frac{r}{t}}^{2\delta}  \nu(B(x,t))^{-2} \abs{ \int_B  \exp\eee{-\frac{d(x,y)^2}{ct^2}} |a(y)| d\nu(y)}^2 \frac{dt}{t}\\
&\leq C r^{2\delta} \nu(B(x, d(x,y_0)))^{-2} \norm{a}^2_{L^1(\nu)} \int_0^{\8} \eee{1+\frac{d(x,y_0)}{t}}^{2n}
\exp\eee{-\frac{d(x,y_0)^2}{c\green{t^2}}} \frac{dt}{t^{1+2\de}}\\
&\leq C \frac{r^{2\delta}}{d(x,y_0)^{2\delta}} \nu(B(x, d(x,y_0)))^{-2} \norm{a}^2_{L^1(\nu)}.
}
Notice that $\delta \leq 1$ and $n\geq 1$ so   $p_0 = 1+n/\delta \leq 2$ and $w \in A_2(\nu)$. By the Cauchy-Schwarz inequality we have
$$\norm{a}_{L^1(\nu)} \leq \norm{a}_{L^2_w(\nu)} \eee{\int_B w^{-1}(x) d\nu(x)}^{1/2} \leq C \frac{\nu(B)}{\nu_w(B)}.$$
Summarizing the estimates above we arrive at
\spx{
G_{\ll} a(x) \leq C \frac{r^\de}{d(x,y_0)^\de} \nu(B(x,d(x,y_0))^{-1} \frac{\nu(B)}{\nu_w(B)}.
}
Denote $S_j(B) = 2^{j+1}B \setminus 2^{j}B$. If $x\in S_j(B)$ then $\nu(B(x,d(x,y_0))) \simeq \nu(2^jB)$ and
\sp{\label{GL1}
\norm{G_{\ll} a}_{L^1_w((2B)^c,\nu)}&= \sum_{j\geq 1} \int_{S_j(B)} G_{\ll} a(x) w(x)d\nu(x)\\
&\leq r^\de \frac{\nu(B)}{\nu_w(B)} \sum_{j\geq 1} \int_{S_j(B)} \frac{(2^j r)^{-\de}}{\nu(2^{j}B)}  w(x) d\nu(x) \\
&\leq \frac{\nu(B)}{\nu_w(B)} \sum_{j\geq 1} 2^{-j\de} \frac{\nu_w(2^{j}B)}{\nu(2^{j}B)}\\
&\leq C \sum_{j\geq 1} 2^{-j\de} \frac{\nu_w(2^{j}B)}{\nu_w(B)} \frac{\nu(B)}{\nu(2^j B)}\\
&\leq C \sum_{j\geq 1} 2^{-j\de} \eee{\frac{\nu(2^jB)}{\nu(B)}}^{p_1-1}\\
&\leq C \sum_j 2^{-j(\de -n(p_1-1))} \leq C,
}
whence \eqref{ggg} follows.
Here we have used the doubling condition and  Lemma~\ref{le2.1}(ii) for $w\in A_{p_1}(\nu)$, where $p_1<p_0=1+\de/n$.
Recall that $p_1<p_0$ can be chosen by the self-improvement property of $A_{p_0}(\nu)$, see  Lemma~\ref{le2.1}(i).

Now let $f\in  H^1_{at}[ \nu,  \w, w]$ and there is a sequence $(\lambda_j)_j$ in $\ell^1$ and a sequence $(a_j)_j$ of
$[ \nu,  \w, w]$-atoms such that $\sum_j \lambda_j a_j$ converges to $f$ in $L^1_w(\nu)$ with
$\sum_j |\la_j| \leq 2\|f\|_{H^1_{at}[ \nu,  \w, w]}$. So by \eqref{ggg} we have
$$
\left\|\sum_{j=1}^{l} \lambda_j a_j -\sum_{j=1}^{k} \lambda_j a_j \right\|_{H^1_{\ll, w} (X)}\leq \sum_{j=k+1}^l  |\lambda_j|
\norm{G_{\ll} a_j}_{L^1_w(\nu)}\leq C\sum_{j=k+1}^l  |\lambda_j|
$$
whenever $l>k>0.$  Then there exists $g$ in $H^1_{\ll, w} (X)$ such that $\sum_j \lambda_j a_j$ converges to
$g$ in $H^1_{\ll, w} (X)$.  By Theorem \ref{th7.1} from the Appendix we have that $H^1_{\ll, w} (X)=H^1_{\ll, S, w}(X)\subseteq L^1_w(X)$ and we have that $g\in H^1_{\ll, w} (X) \subseteq L^1_w(X).$
 Therefore,   $f=g\in H^1_{\ll, w} (X)$ with
$$
\|f\|_{H^1_{\ll, w} (X)}\leq C\lim\limits_{k\to \infty} \sum_{j=1}^k  |\lambda_j|
\norm{G_{\ll} a_j}_{L^1_w(\nu)}\leq C \sum_{j=1}^\8 |\la_j| \leq C\|f\|_{H^1_{at}[ \nu,  \w, w]}
$$
so
$H^1_{at}[ \nu,  \w, w] \subseteq  H^1_{\ll, w} (X) $. The  proof of Theorem~\ref{main1} is complete.
\end{proof}

\rem{rm_max}{
Under the assumptions of Theorem \ref{main1} there exists $C>0$ such that for every
$[ \nu,  \w, w]$-atom $a$ we have
$$\norm{M_\ll a}_{L^1_w(\nu)} \leq C,$$
where $C>0$ is independent of $a$.
}

\begin{proof}
Denote $M_{\ll} f(x) = \sup_{t>0}\abs{\tt_tf(x)}$. Similarly as before, it is enough to
show $\norm{M_{\ll} a}_{L^1_w(\nu)}\leq C$ with $C$ independent of $a$.

Let $a$ be a $[\nu,\w,w]$-atom, so that $\supp\, a \subseteq B = B(y_0,r)$. Then, since $w\in A_2(\nu)$ and $M_{\ll}$ is bounded on $L^2_w(\nu)$,
$$\norm{M_{\ll} a}_{L^1_w(2B,\nu)}\leq \nu_w(2B)^{1/2} \norm{M_{\ll} a}_{L^2_w(\nu)} \leq C \nu_w(2B)^{1/2} \norm{a}_{L^2_w(\nu)} \leq C.$$

Now, let $x\not\in 2B$, so that $d(x,y)\simeq d(x,y_0) \geq r$ for $y\in B$. We have
\spx{
M_{\ll} a(x)& \leq \sup_{t>r^2} \abs{\tt_t a(x)} + \sup_{t\leq r^2} \abs{\tt_t a(x)} =: E_1 + E_2.
}

Note that in in $E_1$ we have $t>r^2\geq d(y,y_0)^2$, so we can use \eqref{Lips}. Hence, by Definition \ref{atoms-hh},
 \spx{
 E_1 &= \sup_{t>r^2} \abs{\int_X \eee{\frac{\tt_t(x,y)}{\w(y)} - \frac{\tt_t(x,y_0)}{\w (y_0)}}a(y) \w(y) d\nu(y)}\\
&\leq  C \sup_{t>r^2} \int_X\eee{\frac{d(y,y_0)}{\st}}^\de \nu(B(x,\st))^{-1} \exp\eee{-\frac{d(x,y_0)^2}{c' t}} \abs{a(y)}  d\nu(y)\\
&\leq  C r^\de \norm{a}_{L^1(\nu)} \sup_{t>r^2} \eee{ t^{-\de/2}  \nu(B(x,\st))^{-1}  e^{-\frac{d(x,y_0)^2}{c't}}}.
 }
For $E_2$ we use \eqref{LG} for $\tt_t(x,y)$ getting
 \spx{
 E_2
 &\leq  C \sup_{t\leq r^2} \int_X \nu(B(x,\st))^{-1} \exp\eee{-\frac{d(x,y)^2}{c' t}} \abs{a(y)}  d\nu(y)\\
 &\leq  C r^\de \norm{a}_{L^1(\nu)} \sup_{t\leq r^2}\eee{ t^{-\de/2}  \nu(B(x,\st))^{-1}  e^{-\frac{d(x,y_0)^2}{c't}}}.
}
By joining these estimates we arrive at
\spx{
E_1+E_2 &\leq  C r^\de \norm{a}_{L^1(\nu)} \nu(B(x,d(x,y_0)))^{-1} \sup_{t>0}\eee{ t^{-\de/2}
 \frac{\nu(B(x,d(x,y_0)))}{ \nu(B(x,\st))}  e^{-\frac{d(x,y_0)^2}{c't}}}\\
&\leq  C r^\de \norm{a}_{L^1(\nu)} \nu(B(x,d(x,y_0)))^{-1} \sup_{t>0}\eee{ t^{-\de/2}  \eee{1+\frac{d(x,y_0)}{ \st}}^n  e^{-\frac{d(x,y_0)^2}{c't}}}\\
&\leq C \eee{\frac{r}{d(x,y_0)}}^\de \nu(B(y_0,d(x,y_0)))^{-1} \norm{a}_{L^1(\nu)}.
}

The rest of the proof goes exactly as for the Littlewood-Paley operator $G_{\ll}$, see \eqref{GL1}.

\end{proof}

\rem{rm_max2}{
Under the additional assumption $w\in A_1(\nu)$ it is possible to strengthen Remark \ref{rm_max} and prove that $H^1_{at}[\nu,\varphi,w] \subseteq H^1_{L,max,w}(X)$. The proof is based on weak-type boundedness of $M_\ll$ on the space $L^1_w(\nu)$, see for example \cite[Lemma 4.3]{Hofmann_Memoirs}.
}
The relation between $H^1_{at}[\nu,\varphi,w]$ and $H^1_{L,max,w}(X)$ stated in Remarks \ref{rm_max} and \ref{rm_max2} are not required in our main argument, so we will not investigate them further here.

\section{Proof of Theorem \ref{main2}   }
\label{sec_proof3}
\setcounter{equation}{0}

To prove Theorem \ref{main2}, we recall that  $T_t=\exp(-tL)$ and $h$ is the $L$-harmonic function for which \eqref{ULG'} holds. As usual, denote
$$
d\nu(x)  = h^2(x) d\mu(x),
$$
and
$$
\tt_t(x,y)  = \frac{T_t(x,y)}{h(x)h(y)}.
$$
Notice that
\ref{H3} means that $h(x)$ is harmonic for $T_t$. As a consequence
$$\int_X \tt_t(x,y) \, d\nu(y) =1, \quad x\in X,$$
so $\w \equiv 1$ is the harmonic function for $\ll$, see  Section~\ref{ssec_Doob}.

{\bf Proof of $ {H^1_L(X)} \subseteq  H^1_{at}[\mu , h].$}
Let $p_0=1+\de n^{-1}$, where $n$ is the dimension on the space of homogeneous type $(X,d,\nu)$ and $\de$
is the H\"older exponent for $\tt_t(x,y)$, see \eqref{Lips}.
Assume that $f\in H^1_L(X)$ or, equivalently, $\wt{f} :=h^{-1}f \in H^1_{\ll,h^{-1}}(X)$, see Section \ref{ssec_Doob}.
 From Theorem \ref{main1}, we have that $\wt{f} = \sum \la_k \wt{a}_k$, where $\sum_k|\la_k| \simeq \norm{\wt{f}}_{H^1_{\ll,h^{-1}}(X)}$
  and there {exist} balls $B_k$ such that:
\alx{
\supp\, \wt{a}_k \subseteq B_k, \qquad
\norm{\wt{a}_k}_{L^2_{h^{-1}}(\nu)} \leq \nu_{h^{-1}}(B_k)^{-1/2}, \qquad
\int \wt{a}_k(x) d\nu(x) = 0.
}
Then $${f = h \wt{f} = \sum_k \la_k a_k,}$$
where $a_k = h \wt{a}_k$. Obviously, $\supp\, a_k \subseteq B_k$ and  $\int a_k(x) h(x) d\mu(x)  =\int_B \wt{a}_k (x) d\nu(x)= 0$. Moreover,
$$\norm{a_k}_{L^2_{h^{-1}}(\mu)} = \norm{\wt{a}_k}_{L^2_{h^{-1}}(\nu)} \leq \nu_{h^{-1}}(B_k)^{-1/2} = \mu_{h}(B_k)^{-1/2}
$$
as desired.

{\bf Proof of $  H^1_{at}[\mu, h] \subseteq {H^1_L(X)}.$}
Note that if $a$ is {an} atom as in Theorem \ref{main2}, then for $\wt{a} = h^{-1} a$ we have
$$\int_B \wt{a}(x) \, d\nu(x) =0, \qquad \norm{\wt{a}_k}_{L^2_{h^{-1}}(\wt{\mu})} \leq \nu_{h^{-1}}(B_k)^{-1/2}. $$
By  Proposition \ref{prop_Doob} and Theorem \ref{main1},
\begin{eqnarray}\label{gggg}
\norm{a}_{H^1_L(X)} = \norm{h^{-1} a}_{H^1_{\ll, h^{-1}}(X)}\leq C
\end{eqnarray}
{ with a constant  $C$ independent of $a$. We then follow an argument as in Theorem~\ref{main1} to
obtain that $  H^1_{at}[\mu, h] \subseteq {H^1_L(X)}.$}
The proof of  Theorem~\ref{main2} is complete.

As a consequence of Theorem~\ref{main2}, we have the following result.

\begin{Cor}\label{corr}
Assume that $L$ and $h$ satisfy all the assumptions of Theorem \ref{main2}. If $f$ belongs to $H^1_L(X)$ and,
additionally, $h$ is in $L^\8(X)$, then
\eq{\label{cancel}
\int f(x) h(x)\, d\mu(x)=0.
}
\end{Cor}

 \begin{proof}
Let $T_t$,be the semigroup related to $L$. By Theorem \ref{main2} we have $f = \sum_k \la_k a_k$,
where $\sum_k|\la_k|<\8$ and $a_k(x)$ are $[\mu,h]$-atoms. Define
$$f_N(x) =\sum_{k=1}^N \la_k a_k(x).$$
Obviously, $f_N \to f$ in $L^1(\mu)$ and $\int_X f_N(x) h(x) \, d\mu(x)=0$ since $f_N \in L^1(\mu)$ and $h\in L^\8$. It follows that
\eqx{\label{cancc}
\int_X f(x) h(x)\, d\mu(x) =0.
}
\end{proof}

\rem{rem_1234}
{
The assumption that $h$ is bounded is necessary in Corollary \ref{corr}. If $h$ is unbounded then \eqref{cancel}
does not need to hold (or even the integral is not well defined). See Section \ref{sec_examples} for examples.
}

\section{Proof of Theorem \ref{thmB}}\label{sec_BMO}
\setcounter{equation}{0}

We start our discussion with the following lemma.
\lem{lem_dual}{
If $f\in H^1_{at}[\mu, h] $  and $g\in BMO[\mu,h],$ then the {pairing $\langle f, g \rangle$} can be defined and satisfies
\eqx{
\abs{\langle f, g \rangle} \leq C \norm{f}_{H^1_{at}[\mu, h]} \norm{g}_{BMO[\mu,h] }.
}
}

\begin{proof}
Let $a $ be {a} $[\mu, h]$-atom. Obviously, the integral $\int a(x) g(x) d\mu(x)$ does not depend on $c$ when {$g = g_1 + ch$}. Moreover,
\spx{
\abs{\int_X a(x) g(x) d\mu(x)} &\leq \int_X |a(x)| \abs{g(x)-ch(x)} d\mu(x)\\
&= \int_X |a(x)|h(x)^{-1/2} \abs{g(x)-ch(x)} h(x)^{1/2} d\mu(x)\\
&\leq \norm{a}_{L^2_{h^{-1}}(\mu)} \eee{\int_B \abs{g(x) -c h(x)}^2 h(x) d\mu(x)}^{1/2}\\
&\leq C \norm{g}_{\BMO}.
}
Therefore, $\int_X \sum_{j=1}^k \la_j a_j(x) g(x) d\mu(x)$ is a Cauchy sequence and we define
the {pairing $\langle f, g \rangle$} as its limit for {an} arbitrary $f\in H^1_{at}[\mu,h]$.
 \end{proof}

 \medskip

\begin{proof}[Proof of Theorem \ref{thmB}]
   By Lemma~\ref{lem_dual}, it follows that if $g\in \BMO$,  then
$$l_g(f) = {\langle f, g \rangle}$$
is a linear bounded functional on $H^1_{at}[\mu, h]$ with norm at most $C\norm{g}_{\BMO}$.

On the other {hand} let $l$ be a linear functional on $H^1_{at}[\mu, h]$. Without {loss} of generality
we assume that $\norm{l}_{H^1_{at}[\mu, h] \to \CC} \leq 1$.
For fixed $B$ let us define the Hilbert space
$$\hh_B = \set{f \in L^2_{h^{-1}}(\mu|_B) \ : \ \int_B f(x) h(x) d\mu(x) = 0}.$$
Obviously, if $f\in \hh_B$, then $f\in  H^1_{at}[\mu, h]$ with
$$\norm{f}_{H^1_{at}[\mu, h]} \leq \mu_h(B)^{1/2} \norm{f}_{\LLLL}$$
Therefore, by the Riesz representation theorem there is $\wt{\wt{g}}_B$ (defined up to $c h^2(x) \chi_B(x)$) such that
$$l_B(f) = \int_B f(x)  \wt{\wt{g}}(x)_B h^{-1}(x) \, d\mu(x)$$
and
$$\norm{\wt{\wt{g}}_B}_{\LLLL}=\norm{l_B}_{\hh_B \to \CC} \leq \mu_h(B)^{1/2}.$$
Let $c_B$ be a constant chosen so that for the function $\wt{g}_B = \wt{\wt{g}}_B +c_B h^2$ we have $\int_{B_0} \wt{g}_B d\mu = 0$ on some fixed ball $B_0$.
Take {an} increasing family $B_0 \subseteq B_1 \subseteq ...$ of balls {such that $\bigcup_{n\in \NN} B_n = X$}.
Since $\wt{g}_{B_n}$ agrees with $\wt{g}_{B_{n+1}}$ on $B_n$
we have that $\wt{g}_{B_n} - \wt{g}_{B_{n+1}} = c_n h^2$ on $B_n \supseteq B_0$. But the left hand side {has zero integral} on $B_0$, so $c_n=0$.
Define {$g_B = h^{-1} \wt{g}_B$} and
$$g(x) = h^{-1}(x) \lim_{n\to \infty} \wt{g}_{{B_n}}(x).$$
Notice that the limit exists, and $g$ coincides with $g_B$ on a ball $B$. Finally,
\spx{
\eee{\frac{1}{\mu_h(B)} \int_B \abs{g(x)-c_B h(x)}^2 h(x) d\mu(x)}^{\frac{1}{2}} &= \eee{\frac{1}{\mu_h(B)}
\int_B \abs{\wt{\wt{g}}_B(x)h^{-1}(x)}^2 h(x) d\mu(x)}^{\frac{1}{2}}\\
&= \eee{\frac{1}{\mu_h(B)} \norm{\wt{\wt{g}}_B}_{\LLLL}^2 }^{\frac{1}{2}} \leq C.
}
This proves that $g\in \BMO$ and $\norm{g}_{\BMO} \leq C$. Also, $l(f) = \int_X f(x) g(x) d\mu(x)$ whenever $f$ is a finite
 combination of atoms. This ends the second part of the proof.
\end{proof}

\section{Applications}
\label{sec_examples}
\setcounter{equation}{0}

As an illustration of our results we shall discuss several examples. Our main results, Theorems~\ref{main2}
and \ref{thmB}
  can be applied to a wide range of operators  such as:  operators with Dirichlet  boundary
	conditions on some domains in $\Rd$, Schr\"odinger operators, and Bessel operators.

For further references let us notice here that the assumption  $h^{-1} \in A_p(\mu_{h^2})$ from Theorem \ref{main2} is equivalent to
\eq{\label{Apw}
\sup_B \frac{h(B)}{h^2(B)} \eee{\frac{h^{2+\frac{1}{p-1}}(B)}{h^2(B)}}^{p-1} \leq C,
}
where $B$ is a ball, $p>1$, and $h^q(B)=\mu_{h^q}(B)$ for $q>0$.

\subsection{Dirichlet Laplacian on $\Omega\subset {\mathbb R}^n$ }\label{ssec61} \
 One of the main motivations for the present paper is the description of the Hardy spaces corresponding to  the Dirichlet Laplacian.
We believe that the applications of our approach  which we describe in Theorems
 \ref{thm_Dir1} and \ref{thm_Dir2} below provide an illuminating way of understanding the results concerning Dirichlet Laplace operator
 obtained by Auscher, Russ, Chang, Krantz and Stein
 in
 \cite{Auscher_Russ, Chang_Krantz_Stein}.
 Assume that a domain $\Omega$
 (an open and connected subset) in $\Rd$ is given. By $\Delta_\Omega$ we will denote the  Laplace  operator with
 the Dirichlet boundary conditions defined on  $\Omega$. We shall consider two particular classes of the set $\Omega$ described
 in Examples 1.1 and 1.2 below.

\subsubsection{Example 1.1: The domain above the graph  of a bounded $C^{1,1}$ function.}
Assume that $\Gamma:\RR^{n-1} \to \RR$ is such that:
\al{
\label{llli1}
&\abs{\nabla \Gamma (x)} \leq C_1,\\
\label{llli2}
&\abs{\nabla \Gamma(x) - \nabla \Gamma (y)} \leq C_2 |x-y|
}
and consider the following domain in ${\mathbb R}^n, n\geq 3$,
\eq{\label{Ome}
\Omega = \set{x\in \Rd \ : \ x_n > \Gamma(x_1,...,x_{n-1})},
}
i.e. the region above the graph  of a bounded $C^{1,1}$ function $ \Gamma.$ One of the main applications of our results is the following theorem.
\thm{thm_Dir1}{
Assume that $\Omega$ is as in \eqref{Ome}, where $\Gamma$ is bounded and satisfies \eqref{llli1}--\eqref{llli2}.
Then there exists a function $h:\Omega \to (0,\8)$, such that
$$h(x) \simeq \mathrm{dist}(x, \Omega^c)$$
and the Hardy space $H^1_{\Delta_\Omega}(\Omega)$ coincides with $H^1_{at}[\mu, h]$, where $\mu$ is the Lebesgue measure on $\Omega$.
}
Theorem \ref{thm_Dir1} is a direct consequence of Theorem \ref{main2}, Proposition \ref{rem_rem},
and Lemma \ref{prop_lip_dom} below. Let us first recall that
  the estimates on the heat kernel $T_t(x,y)$  for the Dirichlet Laplacian $\Delta_{\Omega,D}$ on $\Omega$
  were given in \cite{Song_R_above_Lipshitz}. It was shown there that
 \begin{eqnarray}\label{e6.1}
  T_t(x,y)
\geq C \left({\rho(x)\rho(y)\over t}\wedge 1 \right)  t^{-n/2}   \exp\eee{-\frac{|x-y|^2}{c_1t}},
\end{eqnarray}
and
 \begin{eqnarray}\label{e6.11}
  T_t(x,y)
\leq C \left({\rho(x)\rho(y)\over t}\wedge 1 \right)  t^{-n/2}   \exp\eee{-\frac{|x-y|^2}{c_2t}},
\end{eqnarray}
uniformly for $x, y\in \Omega$ and $t>0$. Here $a\wedge b=\min\{a, b\}$ and $\rho(x) = \mathrm{dist}(x,\Omega^c)$
is the distance between $x$ and $\partial \Omega.$

\lem{prop_lip_dom}{
Let $\Omega$ be a domain given by a bounded $C^{1,1}$ function $\Gamma$, see \eqref{llli1}--\eqref{llli2}.
Then, the function {$\wt{h} = \rho$} defined on $\Omega$ satisfies \ref{H1}--\ref{H4}. Moreover,
for $p>1$ we have $\wt{h}^{-1} \in A_p(\nu)$, where $d\nu(x) = \wt{h}^2(x) dx$ on $\Omega$.
}
\begin{proof} From \eqref{e6.1} and \eqref{e6.11}, we see that
\eq{\label{two-Omega}
\frac{C^{-1}}{t+\rho(x)\rho(y)} t^{-n/2} \exp\eee{-\frac{|x-y|^2}{c_1t}} \leq \frac{T_t(x,y)}{\rho(x)\rho(y)}
\leq  \frac{C}{t+\rho(x)\rho(y)} t^{-n/2} \exp\eee{-\frac{|x-y|^2}{c_2t}}.
}
First, we claim that
\eq{\label{measss}
\nu(B(x,r)) \simeq r^n (r+\rho(x))^2.
}
To prove the claim observe that for $y\in B(x,r)$ we have $\rho(y) \leq \rho(x)+r$, which immediately gives the upper bound.
To see the lower bound recall that $C_1$ is the constant from \eqref{llli1} and consider the set
$$S = B(x,r) \cap \set{y\in \Rd \ : \ y_n \geq  x_n+r/2+C_1|(x_1,...x_{n-1})-(y_1,...,y_{n-1})|}.$$
Observe that $|S|\simeq r^n$ and $S\subseteq \Omega$. Moreover, if $y\in S$ then $\rho(y) \simeq (r+\rho(x))$ and, consequently
we get the lower estimate from \eqref{measss}.

The doubling condition \ref{H1} for $(\Omega, \rho^2(x) dx)$ follows from \eqref{measss}. Moreover, \ref{H4}
is a consequence of \eqref{two-Omega}, \eqref{measss} and the estimate
$$\frac{C^{-1}}{\max(\nu(B(x,\st)),\nu(B(y,\st)))}\leq \frac{t^{-n/2}}{t+\rho(x)\rho(y)} \leq \frac{C}{ \min(\nu(B(x,\st)),\nu(B(y,\st)))}. $$

Similarly to \eqref{measss} we can prove that for $q>0$ we have
\eq{\label{meassss}
h^q (B(x,r)) \simeq r^n (r+h(x))^q,
}
where $h^q(B)$ is the measure with the density $h^q(x) \, dx$ on $\Omega$. Then $h^{-1} \in A_p(\nu)$ for all $p>1$
 follows from \eqref{meassss} and \eqref{Apw}.
\end{proof}

\subsubsection{Example 1.2: Exterior domain outside bounded convex $C^{1,1}$ set.}

Assume that $\Omega\subset {\mathbb R}^n$ is the  exterior of a $C^{1,1}$ compact convex domain, which means that $\Omega^c$
is convex, bounded, and its boundary is locally a $C^{1,1}$ function, see \eqref{llli1}--\eqref{llli2}.
\thm{thm_Dir2}{
Assume that $\Omega$ is the  exterior of a $C^{1,1}$ compact convex domain with boundary that is locally $C^{1,1}$,
see \eqref{llli1}--\eqref{llli2}. Then there exists a function $h:\Omega \to (0,\8)$, such that
$$h(x) \simeq \min(1, \mathrm{dist}(x, \Omega^c))$$
and the Hardy space $H^1_{\Delta_\Omega}(\Omega)$ coincides with $H^1_{at}[\mu, h]$, where $\mu$ is the Lebesgue measure on $\Omega$.
}
{Theorem \ref{thm_Dir2}} is a direct consequence of Theorem \ref{main2}, Proposition \ref{rem_rem},
and Lemma \ref{prop_Dir_h} below. In \cite{Zhang}  the following estimates were proven on the heat kernel $T_t(x,y)$  for the Dirichlet {Laplacian:}
 \begin{eqnarray}\label{e6.2}
 T_t(x,y)
\geq C_1 \left({\rho(x) \over \st \wedge 1}\wedge 1 \right) \left({\rho(y) \over \st \wedge 1}\wedge 1 \right) t^{-n/2}  \exp\eee{-\frac{|x-y|^2}{c_1t}},
\end{eqnarray}
and
\begin{eqnarray}\label{e6.22}
  T_t(x,y)
\leq C_2 \left({\rho(x) \over \st \wedge 1}\wedge 1 \right) \left({\rho(y) \over \st \wedge 1}\wedge 1 \right)  t^{-n/2}  \exp\eee{-\frac{|x-y|^2}{c_2t}},
\end{eqnarray}
uniformly for $x, y\in \Omega$ and $t>0$, where $\rho(x)= \mathrm{dist}(x,\Omega^c)$. For $x\in \Omega$ define
\eq{\label{homega}
\wt{h}(x) = \min(1, \rho(x)).
}
\lem{lem_hhh}{
On $\Omega$ denote the measure $\sigma_q$ that has the density $\wt{h}^q(x) dx$. Then
$$\sigma_q(B(x,r)) \simeq
\begin{cases}
r^n & \text{if } r\geq 1 \text{ or } \rho(x)\geq 1\\
r^n (r+\rho(x))^q & \text{if }r\leq 1 \text{ and } \rho(x)\leq 1
\end{cases}
.$$
In particular, for $\nu := \sigma_2$ we have
$$\nu(B(x,r)) \simeq
\begin{cases}
r^n & \text{if } r\geq 1 \text{ or } \rho(x)\geq 1\\
r^n (r+\rho(x))^2 & \text{if }r\leq 1 \text{ and } \rho(x)\leq 1
\end{cases}
.$$
}

\pr{[Sketch of the proof.] First, observe that since $\Omega^c$ is convex, then
$$|B(x,r)| = |\set{y\in \Omega \ : \ d(x,y)<r}| \simeq r^n.$$
Moreover, if $\rho(x) \geq 1$ or $r\geq 1$, then on substantial part (i.e. on the
 set with measure $\simeq r^n$) of the set $\Omega \cap B(x,r)$ the measure $\nu$ is just the Lebesgue measure.
 In the opposite case, i.e. $r\leq 1$, $\rho(x) \leq 1$ we are close to boundary and $\wt{h}(y) \simeq \rho(y)$.
 Then, the lemma follows by considering two cases: $\rho(x) \geq 2r$ and $\rho(x)\leq 2r$. The details are left to the reader.
}

\lem{prop_Dir_h}{
The function $\wt{h}$ from \eqref{homega} satisfies \ref{H1}--\ref{H4}. Moreover, if $d\nu(x) = \wt{h}^2(x) dx$
on $\Omega$ then for any $p>1$ we have $\wt{h}^{-1} \in A_p (\nu)$.
}

\pr{
Observe first, that from \eqref{Apw} and Lemma \ref{lem_hhh} we have that $\wt{h}(x)$ {satisfies} \ref{H1}
and the $A_p$ condition. Now, we shall show \eqref{ULG'} for $\wt{h}(x)$. The estimates \ref{H4} will
follow from \eqref{e6.2}--\eqref{e6.22} provided that we prove
\spx{\label{enough}
\frac{\left({\rho(x) \over \st \wedge 1}\wedge 1 \right)\left({\rho(y) \over \st \wedge 1}\wedge 1 \right)}
{(\rho(x)\wedge 1)(\rho(y)\wedge 1)}t^{-n/2} \leq \frac{C}{ \min(\nu(B(x,\st)),\nu(B(y,\st)))},\\
\frac{\left({\rho(x) \over \st \wedge 1}\wedge 1 \right)\left({\rho(y) \over \st \wedge 1}\wedge 1 \right)}
{(\rho(x)\wedge 1)(\rho(y)\wedge 1)}t^{-n/2} \geq \frac{C}{ \max(\nu(B(x,\st)),\nu(B(y,\st)))}.
}
Let us notice that we are proving gaussian-type estimates on a doubling space, so we are equally fine with
either $\nu(B(x,\st))$ or $\nu(B(y,\st))$. Recall that the estimates on $\nu(B(x,r))$ are given in Lemma \ref{lem_hhh}.
Obviously, when $t\geq 1$ there is nothing to prove, so let us assume that $t\leq 1$ and denote
$$W = \frac{\left({\rho(x) \over \st }\wedge 1 \right)\left({\rho(y) \over \st }\wedge 1 \right)}{(\rho(x)\wedge 1)(\rho(y)\wedge 1)}t^{-n/2}.$$
By symmetry we shall always consider $x,y$ such that $\rho(x) \leq \rho(y)$. We claim that
\eq{\label{eqqqq}
C^{-1} \nu(B(y,\st))^{-1} \leq W \leq C  \nu(B(x,\st))^{-1}.
}
The claim follows by a careful analysis of the cases:
\ite{
\item $\st\leq 1\leq \rho(x), \rho(y)$,
\item $\st\leq \rho(x)\leq 1\leq \rho(y)$,
\item $\st\leq \rho(x)\leq \rho(y)\leq 1$,
\item $\rho(x) \leq \st\leq 1\leq \rho(y)$,
\item $\rho(x) \leq \st \leq \rho(y)\leq 1$,
\item $\rho(x) \leq \rho(y)\leq \st \leq 1$.
}
{The proof of Lemma \ref{prop_Dir_h} will be finished when we prove the estimate \eqref{eqqqq} in each case.
This follows easily from Lemma \ref{lem_hhh}. Here we will only present one case and leave the others to the reader.
 Assume then that $\rho(x) \leq \st \leq \rho(y)\leq 1$. In this case we have
$$W \simeq \rho(y)^{-1} t^{-(n+1)/2}, \quad \nu(B(y,\st))^{-1} \simeq \rho(y)^{-2}t^{-n/2}, \quad \nu(B(x,\st))^{-1} \simeq t^{-(n+2)/2}$$
and \eqref{eqqqq} follows.
}
}

\begin{Rem} \label{rem_Dir} When  $\Omega$ is the upper-half space ${\mathbb R}^n_+= \set{x=(x', x_n)\in \Rd \ : \ x_n > 0}$,
it follows  by the reflection method, see for example \cite[(6) p. 57]{Strauss}, that
the  heat kernel $T_t(x,y)$  related  to Dirichlet Laplacian $\Delta_{{\mathbb R}^n_+, D}$  on ${\mathbb R}^n_+$
satisfies
\begin{eqnarray*}
  T_t(x,y) =
\frac{1}{ (4\pi t)^{\frac{n}{2}}}  e^{-\frac{|x'-y'|^2}{ 4t}}
\Big(e^{-\frac{|x_n-y_n|^2}{ 4t}} -
e^{-\frac{|x_n+y_n|^2}{ 4t}}\Big),
\label{e2dkern}
\end{eqnarray*}
 for $n \ge 2$. In this case,  the function $h(x)$ from Theorem~\ref{thm_Dir1} equals  $x_n$, see \cite[p. 6]{Gyrya-Saloff-Coste}.
\end{Rem}
{Notice that in for $\Omega = \RR_+^n$ as in Remark \ref{rem_Dir} the conclusion of Theorem \ref{main2}
describes the Hardy space by different atoms to the ones known from \cite{Chang_Krantz_Stein, Auscher_Russ}.
Here all the atoms are as in Definition \ref{atoms-h} with $h(x) = x_n$, whereas in \cite{Chang_Krantz_Stein, Auscher_Russ}
the atoms were different, see \cite[Proposition 1.5]{Chang_Krantz_Stein} and \cite[Theorem 1]{Auscher_Russ}. Nevertheless, all these definitions are equivalent that is  they  give different descriptions of the same space.

}

\begin{Rem} When  $\Omega$ is the  space ${\mathbb R}^n\backslash {\overline B(0, 1)}= \set{x\in \Rd \ : \sum_{i=1}^n x_i^2>1}$, it is known that
 $h(x)={\rm log}|x|$ if $n=2$ and $h(x)=1-|x|^{-n+2}$ if $n>2,$ see \cite[p. 6]{Gyrya-Saloff-Coste}.
\end{Rem}

\begin{Rem} Note that the statements of
Theorems \ref{thm_Dir1} and
\ref{thm_Dir2} essentially coincide saying that $H^1_{\Delta_\Omega}(\Omega)= H^1_{at}[\mu, h]$, where $h$
is the positive harmonic function equal to zero on the boundary of $\Omega$. Note however, that in Theorem \ref{thm_Dir1}
 the function $h$ is unbounded whereas $h\in L^\infty(\Omega)$ in Theorem \ref{thm_Dir2}. Hence in these two settings
 the nature of $H^1_{at}[\mu, h]$ is different in a way {described} in Corollary \ref{corr} and Remark.
\ref{rem_1234}.
\end{Rem}

\subsection{ Schr\"odinger operators} \
Consider $X=\Rd$ with the Lebesgue measure and the Schr\"odinger operators
$$L_V = -\Delta + V,$$
where $V\geq 0$ and $V\in L^1_{loc}(\Rd)$. Since we assume $V\geq 0$ then by the Feynmann-Kac formula we always have \eqref{UG}.
However, the semigroup kernel can be much smaller than classical heat kernel due to the influence of the potential. The Hardy spaces
for $H^1_{L_V}(\Rn)$ were intensively studied, see e.g. \cite{DZ_Revista2, DP_Potential, DZ_JFAA, DZ_Studia_DK, Hofmann_Memoirs}. It appears
that geometric conditions on atoms depend heavily on the dimension $n$ and size of the potential $V$. Let us recall two examples.

\subsubsection{Example 2.1: Potentials from a Kato class.}
Let $n\geq 3$. Then it is known that $T_t = \exp(-tL_{V})$ satisfies \eqref{LG} if and only if $V$ is such that
$$\Delta^{-1} V \in L^\8 (\Rd),$$
see \cite{DZ_JFAA} for details. Then a harmonic function $\w_V$ for $L_{V}$ is given by the formula
$$\w_{V}(x) = \lim_{t\to \8} \int_{\Rd} T_t(x,y) \, dy$$
and satisfies \eqref{double-bounded-h}. In this case for $f\in H^1_{L_{V}}(\Rd)$ we always have $\int f(x) \w_{V}(x) dx = 0$.
For details see \cite{DZ_JFAA, Dziubanski_Preisner_Annali_2017, DZ_Potential_2014}.

\subsubsection{Example 2.2: Inverse square potential.}
 Consider the inverse square potential $V(x) = \gamma |x|^{-2}$ on $\Rd$ with $\gamma>0$ and $n\geq 3$.
 This operator was studied in several papers, see for example \cite{ Killip_Visan, Liskevich_Sobol, Milman_Semenov, Ouhabaz}.
Specifically,  the Hardy space related to this operator was studied in \cite{DZ_Studia_DK, DP_Potential}. The space $H^1_{L_{V}}(\Rd)$ has local
character in the sense that  atoms are either classical atoms or local atoms of the type $|Q|^{-1}\chi_Q$ for some family
of cubes $Q$, see \cite{DZ_Studia_DK}. Obviously, for $f\in H^1_{L_{V}}(\Rd)$ there cannot be a general cancellation condition
like in \eqref{cancel} with any nontrivial function~$h$.

However, as we shall see, there is also another approach to atomic decompositions for $H^1_{L_{V}}(\Rd)$. Consider {the function}
$$h_{V}(x) = |x|^\tau,$$
where
$$\tau = \frac{\sqrt{(n-2)^2+4\gamma}-(n-2)}{2}>0.$$
It appears that $h_V$ is strictly related to the analysis of $T_t=\exp(-tL_{V})$, see \cite{Ouhabaz}.
\lem{prop_Ou}{
The function $h_{V}$ satisfies \ref{H1}--\ref{H4}.
}
\pr{
{For the measure $d\sigma_\be(x) =|x|^\be dx$, $\be>0$, we have
\eq{\label{me_sigma}
\sigma_\be(B(x,r))= r^n(|x|+r)^\be.
}
}
{The doubling condition for $d\sigma_{2\tau}(x) =h_{V}(x)^2\, dx$ follows easily from \eqref{me_sigma}.}
The semigroup related to $L_V$ satisfies   \eqref{ULG'}, see for example  \cite[Theorem 1.2]{Ouhabaz}. See also \cite{Liskevich_Sobol, Milman_Semenov}.
}
\lem{A1RH2ex2}{
For $p>1$ we have $h_{V}^{-1} \in A_p(\nu)$, where $d\nu(x) = h_{V}(x) ^2 dx$ on $\Rd$.
}
\pr{
{It is enough to verify  the condition \eqref{Apw} and this follows from \eqref{me_sigma}.}
}

\cor{cor_Sch}{
Let $n\geq 3$, $V(x) = \gamma |x|^{-2}$, $h_{V}(x) = |x|^\tau$, where $\tau = ( \sqrt{(n-2)^2+4\gamma}-(n-2))/2>0$.
Then the spaces $H^1_{L_{V}}(\Rd)$ and $H^1_{at}[\mu, h_{V}]$ coincide and have equivalent norms. Here $\mu$ is the Lebesgue measure on $\Rd$.
}

\subsection{Bessel operators} \

For $\a>-1$ and $\a\neq 1$ on $X=(0,\8)$ we consider the Euclidean distance and the measure $d\mu(x) = x^\a\, dx$.
The Bessel differential operator is given by
$$L_B f(x) = -f''(x) - \frac{\a}{x}f'(x), \qquad x>0.$$
Observe that a function $h$ satisfies $L_Bh=0$ if
$$h_B(x) = C_1 + C_2 x^{1-\a}.$$

\subsubsection{Example 3.1: Dirichlet Laplacian on $(0,\8)$.}\label{ex-lap}

Let us start with very basic example: $\a=0$, $X=(0,\8)$, and $L_{(0,\8),D}f = -f''$ with Dirichlet boundary condition at $x=0$.
The semigroup generated by this operator has the integral kernel
$$T_t(x,y) = (4\pi t)^{-1/2} \eee{\exp\eee{-\frac{(x-y)^2}{4t}}-\exp\eee{-\frac{(x+y)^2}{4t}}},$$
where $x,y,t>0$. Obviously, the space $H^1_{L_{(0,\8), D}}(X)$ is well studied, see e.g.
\cite{Chang_Krantz_Stein, Auscher_Russ} {(one could also use the results from  \cite{2018arXiv181006937K} and \cite{KP_2021arXiv}
 to get atomic and Riesz transform characterizations of $H^1_{L_{(0,\8), D}}(X)$).} In particular, $H^1_{L_{(0,\8), D}}(X)$ can be
  described by atomic decompositions, where atoms are
either classical atoms on $(0,\8)$ or local atoms of the type $a(x) = |I_m|^{-1} \chi_{I_m}(x)$, $I_m =(2^m,2^{m+1})$, $m\in \ZZ$.

On the other hand, our results provide a new atomic description of $H^1_{L_{(0,\8), D}}(X)$. The  (unbounded) harmonic function for $L_{(0,\8),D}$ is simply
$$h(x) = x.$$
Let $\nu$ be the measure on $(0,\8)$ with the density $x^2 dx$. One can easily check that
\spx{
\tt_t(x,y)=\frac{{T_t(x,y)}}{h(x)h(y)} &\simeq \nu(B(\sqrt{xy}, \sqrt{t}))^{-1} \exp\eee{-\frac{(x-y)^2}{4t}}\\
&\simeq \nu(B(x, \sqrt{t}))^{-1} \exp\eee{-\frac{(x-y)^2}{ct}},
}
so \ref{H4} holds. It is also easy to verify that \ref{H3}--\ref{H1} hold and $x^{-1} \in A_p(\nu)$ for every $p>1$.
As a result of Theorem \ref{main2} we have the following.

\cor{cor_Dir_Lap}{
If a function $f$ belongs to $H^1_{L_{(0,\8), D}}(X)$, then there exist{: $\la_k$, $a_k$, and intervals $B_k$,} such that  $f = \sum_k \la_k a_k$,
$\sum_k |\la_k| \simeq \norm{f}_{H^1_{L_{(0,\8), D}}(X)}$ and $a_k$ are atoms that satisfy:
\alx{
\circ \quad &\supp \,a_k \subseteq B_k,\\
\circ \quad &\eee{\int_{B_k} |a_k(x)|^2 \frac{dx}{x}}^{1/2} \leq \eee{\int_{B_k} x\, dx}^{-1/2}\\
\circ \quad &\int_{B_k} a_k(x)\,  x \, dx =0.
}
}
In other words, an atomic Hardy space with two types of atoms: global (with cancellations) and local (without cancellations) can be
described in a different, more uniform way, where all the atoms have cancellation condition, but w.r.t. a different,
unbounded harmonic function. In Appendix we provide a sketch of a direct proof of the equality of these two atomic spaces.

\subsubsection{Example 3.2: Bessel operator on $(0,\8)$ with Neumann boundary condition at $x=0$.}
Probably the most natural case is to consider $L_{B,(0,\8),N}$, i.e. the operator $L_B$ with Neumann boundary condition at $x=0$.
If $n:=\a+1 \in \NN$ then the analysis of $L_{B,(0,\8),N}$ is equivalent to the analysis of the radial part of the
Laplacian $-\Delta$ on $\Rd$. However, $L_{B,(0,\8),N}$ can be considered also for non-integer $\a$'s. For the results on the
Hardy spaces related to $L_{B,(0,\8),N}$ the reader is referred to \cite{BDT_d'Analyse, Preisner_JAT, DPW_JFAA}. Let us only
 mention briefly, that
 $$h_{B,(0,\8),N}(x) =1$$
 is the harmonic function and all atoms have cancellation of the form:
 $\int a(x) \, d\mu(x) =0$. In other words, the Hardy space $H^1_{(L_{B,(0,\8),N})}(X)$ is the geometric Hardy space in the sense of \cite{CoifmanWeiss_BullAMS}.

\subsubsection{Example 3.3: Bessel operator on $(0,\8)$ with Dirichlet boundary condition at $x=0$.}
 Now, consider $L_{B,(0,\8),D}$, i.e. the space with the Dirichlet boundary condition at $x=0$.
  This example coincides with {Example 3.2} if $\a>1$. However, for  $\a\in (-1,1)$ the function
$$h_{B,(0,\8), D}(x) =x^{1-\a}$$
is unbounded and harmonic for $L_{B,(0,\8),D}$. The crucial estimates \ref{H4} follow
from {\cite[Theorem 5.11]{Gyrya-Saloff-Coste}}. The rest of assumptions of Theorem \ref{main2} is a direct calculation.

\subsubsection{Example 3.4: Bessel operator on $(1,\8)$ with Dirichlet boundary condition at $x=1$.}
Let $\a>-1$, $\a\neq 1$, and $L_{B,(1,\8), D}$ be the Bessel operator with Dirichlet boundary condition at $x=1$.
In the case $n:=\a+1\in \NN$ one can think about Brownian motion on $\Rd \setminus B(0,1)$ killed when entering unit ball. Here
$$h(x) = |1-x^{1-\a}|$$
is the harmonic function. Observe that for $\a\in(-1,1)$ the function $h$ is unbounded, whereas for $\a>1$ we
 have bounded {$h$}, but $\lim_{x\to 1^+} h(x) =0$. Similarly as in the previous example, the estimates \ref{H4}
 follow from  {\cite[Theorem 5.11]{Gyrya-Saloff-Coste}} and the rest of the assumptions of Theorem \ref{main2} can be verified directly.

\section{Equivalence of different atomic decompositions}\label{appendix}

In some examples our results give a new atomic description of $H^1_L(X)$ even for the operators, for which another simple
 atomic description was known before. Let us explain this phenomena a bit in {the} simple example of Dirichlet Laplacian on
  $(0,\8)$, see Subsection \ref{ex-lap}.

Let $X=(0,\8)$ be a space equipped with the Lebesgue measure and denote $I_m = (2^m,2^{m+1})$,
$m\in \ZZ$. Consider the Dirichlet Laplacian  $L=-\Delta_{X,D}$ on $X$. It is known, see \cite{Chang_Krantz_Stein} and \cite{Auscher_Russ}, that
 if $f \in H^1_{L_{X,D}}(X)$, then $f =\sum_k \la_k a_k$, where $\sum_k |\la_k|\simeq \norm{f}_{H^1_{L_{X,D}}(X)}$ and $a_k$ are either:
\begin{itemize}
    \item {\it $\a_1$-atoms}: classical atoms on $(0,\8)$, i.e. for $a$ there exists an interval $B$ such that:
$$\supp \, a \subseteq B, \qquad \norm{a}_{L^2(X)} \leq |B|^{-1/2}, \qquad \int a(x) \, dx = 0$$
or
\item {\it $\a_2$-atoms}: local atoms of the form $a = |I_m|^{-1} \chi_{I_m}$, $m\in \ZZ$.
\end{itemize}
On the other hand $h(x) = x$ is $L$-harmonic and $h(0)=0$. One can easily prove that the kernel
of the semigroup $\exp(-tL_{X,D})$  satisfies \eqref{ULG'} and the measure $x^2\, dx$ is doubling on $(0,\8)$.
 Moreover, using \eqref{Apw}, on easily verifies that $h^{-1} \in A_p(x^2dx)$ for any $p>1$. As a consequence,
 from Theorem \ref{main2} we deduce that each $f\in H^1(L_{X,D})$ can be written as $f = \sum_k \la_k b_k$,
 where $\sum_k |\la_k|\simeq \norm{f}_{H^1_{L_{X,D}}(X)}$ and $b_k$ are
\begin{itemize}
    \item {\it $\be$-atoms}: for $a$ there exists a ball $B$ such that:
$$\supp\, b \subseteq B \subseteq I_k, \qquad \norm{b}_{L^2(\frac{dx}{x})} \leq \eee{\int_B x\, dx }^{-1/2}, \qquad \int x  b(x)\, dx = 0.$$
\end{itemize}

At a first glance it may be surprising, that $H^1_{L_{X,D}}(X)$ has these two different atomic decompositions.
 However, recall that since $h(x)=x$ is unbounded, we cannot say that $\int f(x) x\, dx =0$ even if $f = \sum_k \la_k b_k$
 and $\int b_k(x) x\, dx =0$ for each $k$, see Corollary \ref{corr} and Remark \ref{rem_1234}. The purpose of the following lemma
 is to show that $\a_1$-atoms and $\a_2$-atoms can be decomposed into $\be$-atoms and vice versa. However, we shall not consider
 arbitrary atoms, but, for simplicity of the presentation, we shall assume that the support of every atom considered is already
 contained in some dyadic interval $I_m$.
\lem{atomy}{There exists $C>0$ such that:
\begin{enumerate}
    \item [{\it (a)}] if $a$ is an $\a_2$-atom supported in $I_m$, $m\in \ZZ$, then there
	exist: $\be$-atoms $\eee{b_k}_{k\geq 0}$ and numbers $\eee{\la_k}_{k\geq 0}$ such
	that $a = \sum\limits_{k=0}^{\8} \la_k b_k$ and $\sum_{k=0}^{\8} |\la_k| \leq C$,
    \item [{\it (b)}] if $a$ is an $\a_1$-atom supported in $B \subseteq I_m$, $ m \in \ZZ$,
	then there exist: $N\in \NN$, $\be$-atoms $\eee{b_k}_{k=0}^N$, $\a_2$-atom $a_{N+1}$, and
	numbers $\eee{\la_k}_{k=0}^{N+1}$ such that $a = \sum\limits_{k=0}^{N} \la_k b_k + \la_{N+1}a_{N+1}$ and $\sum\limits_{k=0}^{N+1} |\la_k| \leq C$,
    \item [{\it (c)}] if $b$ is a $\be$-atom supported in $B \subseteq I_m$, $m\in \ZZ$,
	then there exist: $N\in \NN$, $\a_1$-atoms $\eee{a_k}_{k=0}^N$, $\a_2$-atom $a_{N+1}$, and
	numbers $\eee{\la_k}_{k=0}^{N+1}$ such that $b = \sum\limits_{k=0}^{N+1} \la_k a_k$ and $\sum\limits_{k=0}^{N+1} |\la_k| \leq C$.
\end{enumerate}
}
\begin{proof}
To prove {\it (a)} consider $a = |I_m|^{-1}\chi_{I_m}$ for some $m\in \ZZ$. Write
$$a(x) = \sum_{k=0}^\8 2^{-k} b_k(x)$$
where
$$ 2^{-k} b_k(x) = \tau_k \chi_{I_{m+k}}(x) - \tau_{k+1} \chi_{I_{m+k+1}}(x), \quad k=0,1,... \ .$$
Fix $\tau_0 = |I_m|^{-1} = 2^{-m}$. Recursively, we define
$$\tau_{k+1} = \tau_k \frac{\int_{I_k} x\, dx}{\int_{I_{k+1}} x\, dx } = \tau_k \frac{(2^{k+1})^2 - (2^{k})^2}{(2^{k+2})^2 - (2^{k+1})^2} = \frac{\tau_k}{4}$$
so that $\int b_k(x) x\, dx = 0$ for $k=0,1,...$ and $\tau_k = 2^{-m} 4^{-k}$.
Observe that $\supp \, b_k \subseteq I_{m+k}\cup I_{m+k+1}$ and
\spx{
\norm{b_k}_{L^2\eee{\frac{dx}{x}}} &\simeq C 2^k |\tau_k| \eee{\int_{I_{m+k}\cup I_{m+k+1}}\frac{dx}{x}}^{1/2} \leq C 2^k 2^{-m}4^{-k}\simeq C2^{-m-k} \\
&\simeq \eee{\int_{I_{m+k}\cup I_{m+k+1}} x\, dx}^{-1/2}. }
Therefore, $C^{-1} b_k$ are $\be$-atoms and {\it (a)} is proved.

To prove {\it (b)} assume that $a$ is an $\a_1$-atom supported in $B \subseteq I_m$ with some $m$.
Fix a sequence of intervals $B=Q_0 \subseteq Q_1 \subseteq ... \subseteq Q_N = I_m$,
where $|Q_{k+1}|/|Q_k|\leq 2$ and $2^N \simeq |I_m|/|B|$.  Write
$$a(x) = \sum_{k=0}^{N} \la_k b_k(x) + \la_{N+1} a_{N+1}(x)$$
where
\alx{
\la_k &= 2^{k-m} |B|, \quad k=0,...,N+1,\\
\la_0 b_0(x) &= a(x) - \tau_0 \chi_{Q_0}(x),\\
\la_k b_k(x) &= \tau_{k-1} \chi_{Q_{k-1}} - \tau_k \chi_{Q_k}(x), \quad k=1,...,N\\
\la_{N+1} a_{N+1} &= \tau_{N} \chi_{Q_N}
}
and $\tau_k$'s will be specified later on. Observe first that $\sum_{k=0}^{N+1} |\la_k| \simeq 2^{N-m}|B|\simeq  C$.
 Let $y_0$ be the center of $B \subseteq I_m$ and recall that $\int a(x) \, dx =0$. Choose $\tau_k$
  so that $\int x b_k(x) \, dx =0$ for $k=0,...,N$. In particular
\alx{
\tau_0 &= \eee{\int_B x dx }^{-1} \int a(x) x dx\\
\abs{\tau_0} &= \eee{\int_B x dx }^{-1} \abs{\int a(x) (x-y_0) dx} \leq (2^m|B|)^{-1} |B| \int |a(x)| dx \leq 2^{-m}
}
and, for $k=1,...,N$,
\alx{\tau_{k} &= \tau_{k-1} \frac{\int_{Q_{k-1}} x dx} {\int_{Q_k} x dx},\\
|\tau_k| &\leq |\tau_{k-1}|.
}

What is left is to check that $a_{N+1}/C$ is an $\a_2$-atom and $a_k/C$ are $\be$-atoms for $k=0,...,N$.
It is clear that $a_{N+1} = \la_{N+1}^{-1}\tau_N \chi_{I_m}$ and $|\la_{N+1}^{-1} \tau_N| \leq C 2^{N-m} |B| |\tau_0| \leq C |I_m|^{-1}$.
Moreover, for $k=0,...,N$ we have: $\supp \, b_k \subseteq Q_k \cup Q_{k-1}$ ($Q_{-1} = \emptyset$), $\int x b_k(x)  \, dx =0$
(by the choice of $\tau_k$), and
\spx{
\norm{b_k}_{L^2\eee{dx/x}} &\leq C |\tau_k| \la_k^{-1}  \eee{\int_{Q_k \cup Q_{k-1}} 2^{-m} \, dx}^{1/2} \leq 2^{-m} 2^{m-k}|B|^{-1}  (2^{k-m}|B|)^{1/2}\\
&\simeq  \eee{2^{k+m}|B|}^{-1/2} \simeq \eee{\int_{Q_k \cup Q_{k-1}} x\, dx}^{-1/2}.
}
This shows that $b_k/C$ are $\be$-atoms for $k=0,...,N$ and the proof of {\it (b)} is complete. The proof of {\it (c)}
is essentially the same as the case of {\it (b)}. We leave the details to the interested reader.
\end{proof}

\medskip

\section{Appendix: embedding $H^1_{L, S, w}(X)\subseteq L^1_w(X)$}
\label{appendix2}
\setcounter{equation}{0}

In this section we consider a self-adjoint operator $L$ as in Section \ref{ssec32}, i.e. $(X,d,\mu)$ is a doubling metric-measure space, see \eqref{e1.2d}, $\mu(X)=\8$, and the kernel $T_t(x,y)$ corresponding to the semigroup $\exp(-tL)$ satisfies \eqref{UG}.

Recall that $H^1_{L,w}(X) = H^1_{L,S,w}(X) = H^1_{L,at,M,w}(X)$, see Section \ref{ssec32} and Theorem \ref{Liu_Song}. Our main goal here is to prove that $H^1_{L, w}(X)$ embeds continuously in $L^1_w(X)$ as a Banach space. Notice that although the atoms in Definition \ref{atoms-Lw} satisfy $\norm{a}_{L^1_w(\mu)} \leq C$ the inclusion $H^1_{L, w}(X)\subseteq L^1_w(X)$ is not trivial, since in Definition \ref{atoms-Lw} the Hardy space is defined by an abstract completion of the space $\HH_{L,at,M,w}^1(X)$ and this space is defined by $L^2(\mu)$-convergence of the series $\sum_k \la_k a_k$. Nevertheless, using methods from Auscher, McIntosh and Morris \cite{Auscher_McIntosh_Morris} we shall prove in this appendix the following theorem that is needed in the proof of Theorem \ref{main1}.
\begin{Thm}\label{th7.0}
Let $w\in A_{2}(\mu)$. Then $H^1_{L, S, w}(X) \subseteq L^1_w(\mu)$ and
$$\norm{f}_{L^1_w(\mu)} \leq C \norm{f}_{H^1_{L, S, w}(X)}$$
with $C$ independent of $f$.
\end{Thm}
To prove Theorem \ref{th7.0} recall the following definition. Assume that $E_1$ is a normed space contained in a Banach space $E_2$ with $\norm{x}_{E_2} \leq C \norm{x}_{E_1}$ for $x\in E_1$. We say that {\it $E_1$ has a completion in $E_2$} if there exists a Banach space $E_3$ such that: $E_1\subseteq E_3 \subseteq E_2$, $E_1$ is dense in $E_3$, $\norm{x}_{E_1}=\norm{x}_{E_3}$ for $x\in E_1$, and $\norm{x}_{E_2}\leq C \norm{x}_{E_3}$ for $x\in E_3$. For more comments see \cite[p. 870-871]{Auscher_McIntosh_Morris}. We shall prove the following.
\begin{Prop}\label{th7.1}
Let $w\in A_{2}(\mu)$.  Then the completion  $H^1_{L, S, w}(X)$ of
 $\{f\in L^2(\mu):  \|S_L f\|_{L^1_w(\mu)} < \infty \}
 $ in $  L^1_w(\mu)$ exists.
\end{Prop}

Notice that Theorem \ref{th7.0} follows directly from Proposition \ref{th7.1} and the fact that the completion of a normed space is unique. The rest of this section is devoted to prove Proposition \ref{th7.1}.

First, we recall the notion of the weighted tent spaces on $X$ from  \cite{Song_Wu}, see also
 \cite{CMS, Russ}. For a measurable function $F$
defined on $X\times(0,\infty)$ consider
\begin{eqnarray*}
{\mathcal A} F(x):=\left( \int_0^{\infty}\!\int_{d(y,x)<  t}
|F(y,t)|^2 {d\mu(y)\over \mu(B(x, t))}{dt\over t}\right)^{1/2}.
\end{eqnarray*}
 For $p\in [1, \infty)$ and $w\in A_{\infty}(\mu)$, the tent space
  $T^p_{w}(X)$  is defined as the space
of measurable functions $F$ on $X\times (0,\infty)$ for which
${\mathcal A}F\in L^p_w(\mu)$. This is equipped with
$\|F\|_{T^p_{ w}(X)}:=\|{\mathcal A}F\|_{L^p_w(\mu)}$.
For simplicity we will often write $T^p(X) $ in place of $T^p_{ w}(X)$ with $w=1$.
The tent space $T^{\infty}_{w}(X)$ is the Banach space of all $F$ satisfying
$$
\|F\|_{T^{\infty}_{w}(X)}:=\sup_{x\in X}\sup_{B\in {\mathbb B}(x)}  \left( {1\over \mu_w(B)} \iint_{T(B)} |F(y,t)|^2
{ \mu(B(y, t))\over \mu_w(B(y,t))} d\mu(y) {dt\over t}\right)^{1/2},
 $$
 where ${\mathbb B}(x)$ denotes the set of all balls $B\subseteq X$ with the property that $x\in B$,
 and the tent $T(B)=\{ (y, t)\in X\times (0, \infty): d(y, X\backslash B)\geq t\}.$
  Using a similar argument as in \cite[Lemma 2.5]{CCFY}, where the case $X={\mathbb R^n}$ is considered, it can be verified that for   $w\in A_{\infty}(\mu)$,  the pairing
$$\langle F, G\rangle:= \int_{X\times (0, \infty)}\
 F(x,t){\overline { G(x,t)}} \ {d\mu(x)dt\over t}
 $$
 realizes
  $T^{\infty}_w(X)$ as equivalent with the Banach space dual of $T^1_w(X)$.

 Recall that a measurable function $A$ on
$X\times(0,\infty)$ is said to be a $T^1_{w}$-atom if there exists a ball
$B\subset X$ such that $A$ is supported in the tent $T(B)$
  and $$\|A\|_{T^2_{w}(X)}\leq \mu_w(B)^{-1/2}.$$
It has been proved in \cite[Theorem 1.8]{Song_Wu}, see also \cite[Remark 3.3]{Song_Wu},
that every $F\in T^1_{w}(X)$ has an atomic
decomposition.

\begin{Lem}\label{prop7.6} {Let $w\in A_\8(\mu)$}. For every element $F\in T^1_{w}(X)$ there exist
a  sequence $\{\lambda_j\}_{j=1}^\infty$ and a sequence of
$T^1_{w}$-atoms $\{A_j\}_{j=1}^\infty$ such that $\sum |\lambda_j| < \infty$ and
\begin{equation}\label{eqtentdecomp}
F=\sum_{j}\lambda_jA_j\quad\mbox{in }\,T^1_{w}(X) \,\text{ and a.e. in } X\times (0,\infty).
\end{equation}
Moreover,  $$
\sum_{j}|\lambda_j|\approx \|F\|_{T^1_{w}(X)},$$
where the implicit constants depend only on the homogeneous space properties of $X$.

Finally, if
$F \in T^1_{w}(X)\cap T^2(X)$, then the decomposition (\ref{eqtentdecomp})
also converges in $T^2(X)$.
\end{Lem}

  Let  $K_{\cos(t\sqrt{L})}(x,y)$ denote  the integral kernel
of the operator $\cos(t\sqrt{L})$.
It is known, see for example \cite{S2}, that there exists a constant $c_0>0$ such that for every
$t>0$,
$$
\supp K_{\cos(t\sqrt{L})}\subset {\mathcal D}_t:=
\Big\{(x,y)\in X\times X: \,d(x,y)\leq c_0t\Big\}.
$$
 Let $M\geq 1$, and let $\varphi\in C^{\infty}_0(\mathbb R
 )$ be
even, $\mbox{supp}\,\varphi \subset (-c_0^{-1}, c_0^{-1})$
  with $\varphi \geq c > 0$ on $(-1/(2c_0), 1/(2c_0))$.
Let $\Phi$ denote the Fourier transform of
$\varphi$.  Set $\Psi(x):=x^{2(M+1)}\Phi(x)$,
$x\in{\mathbb{R}}$.
Consider the operator
$\pi_{\Psi,L}: T^2(X)\rightarrow L^2(\mu)$, given by
$$
\pi_{\Psi,L}(F)(x):= \int_0^{\infty}\Psi(t\sqrt L)\big(F(\cdot,\, t)\big)(x){dt\over t},
$$
where the improper integral converges weakly in $L^2(\mu)$. The bound
$$
\|\pi_{\Psi,L}F\|_{L^2(\mu)} \leq C_M \|F\|_{T^2(X)},\,\,\,\,M\geq 0,
$$
 follows readily by duality and the $L^2(\mu)$ quadratic estimate.
 We have the following analogue of the well-known argument of \cite[Proposition 4.10]{Song_Wu}.

\begin{Lem}\label{le7.7} Let $w\in A_{2}(\mu)$.  Suppose that $A$ is a $T^1_{w}$-atom associated to a ball
$B\subset X$ (or more precisely, to its tent $T(B)$).  Then  for every
$M\geq1$, there is a uniform constant $C_M$ such that $C_M^{-1}\, \pi_{\Psi,L} (A)$
is an $(L, M, w)$-atom associated with the concentric ball $2B$.
\end{Lem}

With the above preliminary results, we now start to prove Theorem~\ref{th7.1} by adapting an argument as in
\cite[Lemma 3.9 and Theorem 3.10]{Auscher_McIntosh_Morris}.

\begin{proof}[Proof of Proposition~\ref{th7.1}]
Let $w\in A_{2}(\mu)$ and set $E_{L, S, w}^1(X)=\{f\in L^2(\mu):  \|S_L f\|_{L^1_w(\mu)} < \infty \}$.

\noindent
{\bf Step 1: }
$E_{L, S, w}^1(X) \subseteq L^1_w(\mu)$ and $\norm{f}_{L^1_w(\mu)} \leq C \norm{S_L f}_{L^1_w(\mu)}$ for $f\in E_{L, S, w}^1(X)$.

Let $f\in E_{L,S, w}^1(X)$ and set
$$F(\cdot,t):= t^2Le^{-t^2L}f.$$
We note that $F\in T^2(X)\cap T^1_{w}(X)$, by   the definition of $E_{L, S,  w}^1(X)$.
Therefore, by Lemma~\ref{prop7.6}, we have that
$$F =\sum_j \lambda_j \, A_j,$$
where each $A_j$ is a $T^1_{w}$-atom, the sum converges in both $T^2(X)$ and $T^1_{w}(X)$,
and
$$\sum_j |\lambda_j|\leq C\|F\|_{T^1_{w}(X)}=C\|S_Lf\|_{L^1_w(\mu)}.
$$
Also, by $L^2$-functional
calculus, see \cite{Mc}, and using that $f \in L^2(\mu)$, we have the ``Calder\'{o}n reproducing formula"
\begin{eqnarray*}
f(x)&=&c_{\Psi}\int_0^{\infty}\Psi(t\sqrt L)(t^{2}L e^{-t^2{L}}f)(x) {dt\over t}\\
&=&c_\Psi \,\pi_{\Psi,L}(F) = c_\Psi \,\sum \lambda_j\, \pi_{\Psi,L}(A_j),
\end{eqnarray*}
where
the last sum converges in $L^2(\mu)$ and  $E^1_{L, S, w}(X)$.
Moreover, by Lemma \ref{le7.7}, for every $M\geq 1$, we have that up to multiplication by some
harmless constant $C_M$,
each $a_j := c_\Psi\, \pi_{\Psi,L}(A_j)$
is  an $(L,M, w)$-atom, and so $\|a_j\|_{L^1_w(\mu)}\leq C$ with a constant $C>0$ independent of $j$.
Consequently,
$\sum_j \lambda_j \pi_{\Psi,L}(A_j)$ converges to ${\tilde f}$ in $L^1_w(\mu)$. We must have $f={\tilde f}\in L^1_w(\mu)$
since $L^1_w(\mu)$ and $L^2(\mu)$ are embedded in $L^{1/2}_{{\rm loc}}(\mu)$ (notice that $w\in A_2(\mu)$ implies $\int_B w^{-1}\, d\mu <\8$) and so $\sum_j \lambda_j \pi_{\Psi,L}(A_j)$ converges to $f$ in $L^1_w(\mu)$
with
$$
\|f\|_{L^1_w(\mu)}=\lim\limits_{n\to \infty} \Big\|\sum_j \lambda_j \pi_{\Psi,L}(A_j)\Big\|_{L^1_w(\mu)} \leq C \sum_j|\la_j| \leq
C\|S_Lf\|_{L^1_w(\mu)}= C\|f\|_{E^1_{L, S, w}(X)},
$$
and so  $E_{L,S, w}^1(X)\subseteq L^1_w(\mu)$.

\noindent
{\bf Step 2: }
The completion of $E^1_{L,S,w}(X)$ in $L^1_w(\mu)$ exists.

We shall use the following proposition that states necessary and sufficient condition for the existence
of a completion inside a given Banach space, see \cite[Proposition 2.2]{Auscher_McIntosh_Morris}.

\begin{Lem}\label{prop7.2}
Let $E_1$ and a normed space and suppose that $E_1\subseteq E_2 $ for some Banach space $E_2$,
so the identity $I: E_1\to E_2$ is bounded. The following are equivalent:

\begin{itemize}
\item[(i)]  the completion of $E_1$ in $E_2$ exists;
\item[(ii)]  for each Cauchy sequence $(x_n)_n$ in $E_1$ that converges to $0$ in $E_2$, it follows that
 $(x_n)_n$ converges to $0$ in $E_1$.
\end{itemize}
\end{Lem}

The proof of the  above lemma is a simple functional analysis argument. To complete the proof of Proposition~\ref{th7.1} we consider  $E_1=E^1_{L, S, w}(X)$ and $E_2=L^1_w(\mu)$. From $(ii)$ of Lemma \ref{prop7.2} it is enough to consider $(f_n)_n$ that is a Cauchy sequence in $E^1_{L,S,  w}(X)$ and converges to $0$ in $L^1_w(\mu)$. We claim that
$(f_n)_n$ converges to $0$ in $E^1_{L, S, w}(X)$.
 For all $n$ we have that
$$
\|f_n\|_{E^1_{L, S, w}(X)}=\|S_Lf_n\|_{L^1_w(\mu)} = \|t^2Le^{-t^2L} f_n\|_{T_{w}^1(X)}
$$
and, since $(f_n)_n$ is Cauchy sequence in $E^1_{L, S, w}(X)$, there exists $U$ in $T_{w}^1(X)$ such that $t^2Le^{-t^2L} f_n$
converges to $U$ in $T_{w}^1(X)$.

Denote
$$
\mm=\Big\{ F\in T^2(X)\cap T_{w}^{\infty}(X):    \frac{ \pi_{L}F}{w}\in L^{\infty}(\mu)   \Big\},
$$
where $\pi_{L}F(x):=\int_0^{\infty} t^2L e^{-t^2L} F dt/t$. Following the method from \cite{Auscher_McIntosh_Morris} {we claim that $\mm$} is weak-star dense in $T_w^{\infty}(X)$. {Indeed, one can show that for $F \in T^\8_w(X)$ the functions $F\cdot \chi_{A_n}$ belongs to $\mm$, where $A_n = [1/n, n] \times \{x\in X \ : \ d(x,x_0)<n, \  w^{-1}(x) \leq n\})$, c.f.  \cite[Remark 3.11 and p.882]{Auscher_McIntosh_Morris}. }
Using the duality, we have that for any $F\in \mm,$
\begin{eqnarray*}
|\langle U, F\rangle |&\leq& |\langle U -t^2Le^{-t^2L} f_n, F\rangle | + |\langle  t^2Le^{-t^2L} f_n, F\rangle |\\
&\leq&C \| U -t^2Le^{-t^2L} f_n\|_{T^1_w(X)} \|F\|_{T^{\infty}_w(X)} +\|f_n\|_{L^1_w(\mu)} \|w^{-1}\cdot \pi_{L}F\|_{L^{\infty}(\mu)}.
\end{eqnarray*}
  Moreover,
since $ \| w^{-1}\cdot\pi_{ L} F\|_{L^{\infty}(\mu)}<\infty$ and $\|F\|_{T^{\infty}_w(X)}<\infty$, the preceding convergence results
imply that
$$
\langle U, F\rangle =0, \ \ \ \forall F\in \mm.
$$
Then, since $U\in T^1_w(X)$ and $\mm$ is weak-star dense in $T^{\infty}_w(X)$, it follows
that $\langle U, F\rangle =0$ for all $F\in T^{\infty}_w(X),$  hence $U=0$ and $(f_n)_n$ converges to $0$ in $E^1_{L, S, w}(X)$ as claimed.
This proves that the completion  $H^1_{L, S, w}(X)$ of $E^1_{L,S,  w}(X)$ in $  L^1_w(\mu)$ exists.
Hence,  the  proof of Proposition~\ref{th7.1} is complete.
\end{proof}

\bigskip

\noindent
{\bf Acknowledgements:} The authors would like to thank Laurent Saloff-Coste and the anonymous referees for helpful remarks. AS and MP were
 supported by Australian Research Council  Discovery Grant DP DP160100941 and DP200101065. MP was supported by the grant No. 2017/25/B/ST1/00599 from National Science Centre (Narodowe Centrum Nauki), Poland.
 LY was supported  by the grant No.~11521101 and~11871480  from  the NNSF
of China.

\bibliographystyle{amsplain}        

\def\cprime{$'$}
\providecommand{\bysame}{\leavevmode\hbox to3em{\hrulefill}\thinspace}
\providecommand{\MR}{\relax\ifhmode\unskip\space\fi MR }
\providecommand{\MRhref}[2]{%
  \href{http://www.ams.org/mathscinet-getitem?mr=#1}{#2}
}
\providecommand{\href}[2]{#2}

\end{document}